\documentclass[11pt]{article}
\usepackage{latexsym,amssymb,upref,amsmath,amsthm, amsfonts,authblk}
\usepackage{amssymb,amsmath,amsthm, calc, graphicx}
\usepackage{epsfig}
\usepackage{breqn}
\usepackage{footnpag}
\usepackage{rotating}
\usepackage{amsfonts}
\usepackage{setspace}
\usepackage{fullpage}
\usepackage{enumitem}
\usepackage{bbold}
\usepackage{comment}
\usepackage{pgf,tikz}
\usepackage{mathrsfs}
\usetikzlibrary{arrows}
\usepackage{yhmath}
\usepackage{hyperref}
\usepackage{authblk}
\usepackage{mathrsfs}

\newtheorem{thm}{Theorem}
\newtheorem*{thm*}{Theorem}
\newtheorem*{prop*}{Proposition}

\newtheorem{lemma}[thm]{Lemma}
\newtheorem{conjecture}[thm]{Conjecture}

\newtheorem{cor}[thm]{Corollary}
\newtheorem{corollary}[thm]{Corollary}

\newtheorem{claim}{Claim}

\theoremstyle{remark}
\newtheorem{obs}{Observation}
\newtheorem{remark}{Remark}

\newcommand\cC{{\mathscr C}}

\newcommand\cF{{\mathcal F}}

\newcommand\cH{{\mathcal H}}

\newcommand\cN{{\mathcal N}}

\newcommand\cS{{\mathcal S}}

\newcommand{\abs}[1]{\left\lvert{#1}\right\rvert}
\newcommand{\ignore}[1]{}

\title{Generalized Tur\'an problems for even cycles}

\begin{document}
	
	\author{
		D\'aniel Gerbner \thanks{Alfr\'ed R\'enyi Institute of Mathematics, Hungarian Academy of Sciences. e-mail: gerbner@renyi.hu}
		\qquad
		Ervin Gy\H{o}ri \thanks{Alfr\'ed R\'enyi Institute of Mathematics, Hungarian Academy of Sciences. e-mail: gyori.ervin@renyi.mta.hu}
		\qquad
		Abhishek Methuku \thanks{\'Ecole Polytechnique F\'ed\'erale de Lausanne and Central European University. e-mail: abhishekmethuku@gmail.com}
		\qquad
		M\'at\'e Vizer \thanks{Alfr\'ed R\'enyi Institute of Mathematics, Hungarian Academy of Sciences. e-mail: vizermate@gmail.com.}}
	
	\date{
		\today}

	\maketitle
	
	\begin{abstract}
		Given a graph $H$ and a set of graphs $\cF$, let $ex(n,H,\cF)$ denote the maximum possible number of copies of $H$ in an $\cF$-free graph on $n$ vertices. We investigate the function $ex(n,H,\cF)$, when $H$ and members of $\cF$ are cycles. Let $C_k$ denote the cycle of length $k$ and let $\mathscr C_k=\{C_3,C_4,\ldots,C_k\}$. We highlight the main results below.
       
       \begin{enumerate}

       	\item[(i)] We show that $ex(n, C_{2l}, C_{2k}) = \Theta(n^l)$ for any $l, k \ge 2$. Moreover, in some cases we determine it asymptotically: We show that $ex(n,C_4,C_{2k}) = (1+o(1)) \frac{(k-1)(k-2)}{4} n^2$ and that the maximum possible number of $C_6$'s in a $C_8$-free bipartite graph is $n^3 + O(n^{5/2})$. 
       	
       	\item[(ii)] Erd\H os's Girth Conjecture states that for any positive integer $k$, there exist a constant $c > 0$ depending only on $k$, and a family of graphs $\{G_n\}$ such that $|V(G_n)|=n$, $|E(G_n)|\ge cn^{1+1/k}$ with girth more than $2k$.
       	
       	\hspace{2mm} Solymosi and Wong \cite{SW2017} proved that if this conjecture holds, then for any $l \ge 3$ we have $ex(n,C_{2l},\mathscr C_{2l-1})=\Theta(n^{2l/(l-1)})$. We prove that their result is sharp in the sense that forbidding any other even cycle decreases the number of $C_{2l}$'s significantly: For any $k > l$, we have $ex(n,C_{2l},\cC_{2l-1} \cup \{C_{2k}\})=\Theta(n^2).$
       	More generally, we show that for any $k > l$ and $m\ge 2$ such that $2k \neq ml$, we have $ex(n,C_{ml},\cC_{2l-1} \cup \{C_{2k}\})=\Theta(n^m).$
       	
       	\item[(iii)] We prove  $ex(n,C_{2l+1},\cC_{2l})=\Theta(n^{2+1/l}),$ provided a stronger version of Erd\H os's Girth Conjecture holds (which is known to be true when $l = 2, 3, 5$). 
       	This result is also sharp in the sense that forbidding one more cycle decreases the number of $C_{2l+1}$'s significantly: More precisely, we have $ex(n, C_{2l+1}, \cC_{2l} \cup \{C_{2k}\}) = O(n^{2-\frac{1}{l+1}}),$ and $ex(n, C_{2l+1}, \cC_{2l} \cup \{C_{2k+1}\}) = O(n^2)$ for $l > k \ge 2$.
       	
       	\item[(iv)] We also study the maximum number of paths of given length in a $C_k$-free graph, and prove asymptotically sharp bounds in some cases.
       \end{enumerate}

	\end{abstract}
	
	\vspace{4mm}
	
	
	\noindent
	{\bf AMS Subj.\ Class.\ (2010)}: 05C35, 05C38
	
	\section{Introduction}

The Tur\'an problem for a set of graphs $\cF$ asks the following. What is the maximum number $ex(n,\cF)$ of edges that a graph on $n$ vertices can have without containing any $F \in \cF$ as a subgraph? When $\cF$ contains a single graph $F$, we simply  write $ex(n,F)$. This function has been intensively studied, starting with Mantel \cite{M1907} and Tur\'an \cite{T1941} who determined $ex(n,K_r)$ where $K_r$ denotes the complete graph on $r$ vertices with $r \ge 3$. See \cite{FS2013,S1997} for surveys on this topic.

    Let $C_k$ denote a cycle on $k$ vertices and let $P_k$ denote a path on $k$ vertices. \textit{Length} of a path $P_k$ is $k-1$, the number of edges in it and length of a cycle $C_k$ is $k$. A theorem of Simonovits \cite{S1968} implies that for odd cycles, we have $ex(n,C_{2k+1})=\lfloor n^2/4 \rfloor$ for any $k \ge 1$ and $n$ large enough. For even cycles $C_{2k}$, Bondy and Simonovits \cite{BS1974} proved the following upper bound.
    
    \begin{thm}[Bondy, Simonovits \cite{BS1974}]
    \label{bs} For $k \ge 2$ we have $$ex(n,C_{2k})=O(n^{1+1/k}).$$
    
    \end{thm}
The order of magnitude in the above theorem is known to be sharp only for $k = 2,3,5$. The case when all cycles longer than a given length are forbidden, was considered by Erd\H os and Gallai \cite{EG1959}.
	
	\begin{thm}[Erd\H os, Gallai \cite{EG1959}]\label{erga} If a graph does not contain any cycle of length more than $k$, then it has at most $\frac{(k-1)n}{2}$ edges.
		
	\end{thm}

On the other hand, if all the short cycles are forbidden, Alon, Hoory and Linial \cite{AHL2002} proved the following. To state their result let us introduce the following notation: let $A$ be a set of integers, each at least 3. Then let the set of cycles $\cC_A=\{C_a: a \in A\}$. If $A=\{3,4,...,k\}$ for some integer $k$, then we denote the corresponding set of cycles by $\cC_k$.

\begin{thm}[Alon, Hoory, Linial \cite{AHL2002}]
\label{AlonHL}

For any $k\ge 2$ we have

\vspace{3mm}

(i) $ex(n,\cC_{2k})<\frac{1}{2}n^{1+1/k}+\frac{1}{2}n$,

\vspace{2mm}

(ii) $ex(n,\cC_{2k+1})<\frac{1}{2^{1+1/k}}n^{1+1/k}+\frac{1}{2}n$.
\end{thm}

\noindent 
For more information on Tur\'an number of cycles one can consult the survey \cite{V2016}.

\subsection*{Notation and definitions} 

The girth of a graph is the length of a shortest cycle in it. We say a graph has even girth if its girth is of even length, otherwise we say it has odd girth. Now we introduce basic notation that we will use throughout the paper. 
\begin{itemize} 
\item We will denote by $v_1v_2 \dots v_{k-1}v_kv_1$ a cycle $C_k$ with vertices $v_1,v_2, \dots, v_k$ and edges $v_iv_{i+1}$ \newline($i=1, \dots, k-1$) and $v_kv_1$. Similarly $v_1v_2\dots v_{k-1}v_k$ denotes a path $P_k$ with vertices $v_1,v_2, \dots, v_k$ and edges $v_iv_{i+1}$ ($i=1, \dots, k-1$). 

\item For two graphs $H$ and $G$, let $\cN(H,G)$ denote the number of copies of $H$ in $G$.

\item For a vertex $v$ in $G$, let $N_i(v)$ denote the set of vertices at distance exactly $i$ from $v$.

\item For any two positive integers $n$ and $l$, let $(n)_l$ denote the product $$n(n-1)(n-2) \ldots (n-(l-1)).$$ 

\end{itemize}

%

\subsection{Generalized Tur\'an problems}

Given a graph $H$ and a set of graphs $\cF$, let $$ex(n,H,\cF)=\max_{G} \{\cN(H,G): \text{$G$ is an $\cF$-free graph on $n$ vertices.} \}$$ 
If $\cF=\{F\}$, we simply denote it by $ex(n,H,F)$. This problem was initiated by Erd\H os \cite{E1962}, who determined $ex(n,K_s,K_t)$ exactly. 
Concerning cycles, Bollob\'as and Gy\H ori \cite{BGy2008} proved that $$(1 + o(1)) \frac{1}{3 \sqrt 3} n^{3/2} \le ex(n,C_3,C_5) \le (1 + o(1))  \frac{5}{4} n^{3/2}$$ and this result was extended by Gy\H ori and Li \cite{GyoriLi} showing that $$ex(n,C_3,C_{2k+1}) \le \frac{(2k-1)(16k-2)}{3} \cdot ex(n, C_{2k})$$ for $k \ge 2$.  This was later improved by F\"uredi and \"Ozkahya \cite{FO2017} by a factor of $\Omega(k)$. 

The systematic study of the function $ex(n,H,F)$ was initiated by Alon and Shikhelman in \cite{ALS2016}, where they improved the result of Bollob\'as and Gy\H ori by showing that $ex(n,C_3,C_5) \le  (1 + o(1)) \frac{\sqrt 3}{2}  n^{3/2}.$ This bound was further improved in \cite{C5C3}  and then very recently in \cite{C5C3v2} by Ergemlidze and Methuku who showed that $ex(n,C_3,C_5) <  (1 + o(1)) 0.232  n^{3/2}.$ Another notable result is the exact computation of $ex(n,C_5,C_3)$ by Hatami, Hladk\'y, Kr\' al, Norine, and Razborov \cite{HHKNR2013} and independently by Grzesik \cite{G2012}, where they showed that it is equal to $(\frac{n}{5})^5$. Very recently, the asymptotic value of $ex(n,C_k,C_{k-2})$ was determined for every odd $k$ by  Grzesik and Kielak  in \cite{GK2018}. Concerning paths, Gy\H ori, Salia, Tompkins and Zamora \cite{GySTZ2018} determined $ex(n,P_{l},P_k)$  asymptotically.

%

In \cite{ALS2016}, Alon and Shikhelman characterized the graphs $F$ with $ex(n,C_3,F)=O(n)$ and  more recently, Gerbner and Palmer \cite{GP2017} showed that for every $l\ge 4$ and every graph $F$ we have either $ex(n,C_l,F)=\Omega(n^{2})$ or $ex(n,C_l,F)=O(n)$, and characterized the graphs $F$ for which the latter bound holds. They also showed 

\begin{thm}[Gerbner, Palmer \cite{GP2017}]\label{celebrated} For $t \ge 2$ and $l \ge 4$ we have
$$ex(n,C_{l},K_{2,t})=\frac{1}{2l}(t-1)^{l/2}n^{l/2},\hspace{0.8truecm} ex(n,P_{l},K_{2,t})=\frac{1}{2}(t-1)^{(l-1)/2}n^{(l+1)/2}.$$ 
\end{thm}
Note that the case $t=2$ was proved independently by Gishboliner and Shapira \cite{gs2017}. 

\vspace{2mm}

In this paper, we mainly focus on the case when $H$ is an even cycle of given length and $\cF$ is a family of cycles. 

\vspace{2mm}

The function $ex(n,H,F)$ is closely related to the area of Berge hypergraphs, see e.g.~\cite{GMV2017,GP20172}. Let $k \ge 2$ be an integer. A \textit{Berge cycle} of length $k$ is an alternating sequence of distinct vertices and hyperedges of the form $v_1$,$h_{1}$,$v_2$,$h_{2},\ldots,v_k$,$h_{k}$,$v_1$ where $v_i,v_{i+1} \in h_{i}$ for each $i \in \{1,2,\ldots,k-1\}$ and $v_k,v_1 \in h_{k}$ and is denoted by Berge-$C_k$.~Gy\H{o}ri and Lemons \cite{GyL2012_2} proved the following two theorems.

\begin{thm}[Gy\H{o}ri, Lemons \cite{GyL2012_2}] Let $r \ge 3$ be a positive integer.
If $\cH$ is an $r$-uniform Berge-$C_{2k}$-free hypergraph on $n$ vertices, then it has at most $O(n^{1+\frac{1}{k}})$ hyperedges.

\end{thm}

\begin{thm}[Gy\H{o}ri, Lemons \cite{GyL2012_2}] If $\cH$ is a Berge-$C_{2k}$-free hypergraph on $n$ vertices, such that $|e| \ge 4k^2$ for every hyperedge $e$, then we have: $$\sum_{e \in E(\cH)} |e| =O(n^{1+\frac{1}{k}}).$$

\end{thm}

\noindent 
The previous two theorems easily imply the following corollary that we will use later.

\begin{cor}\label{gyorilemons}

If $\cH$ is a Berge-$C_4$-free hypergraph on $n$ vertices, then we have $$\sum_{e \in E(\cH)} |e| = O(n^{1.5}).$$

\end{cor}

\subsection{Forbidding a set of cycles}

 The famous Girth Conjecture of Erd\H os \cite{E1963} asserts the following.

\begin{conjecture}[Erd\H os's Girth Conjecture \cite{E1963} for $k$] For any positive integer $k$, there exist a constant $c > 0$ depending only on $k$, and a family of graphs $\{G_n\}$ such that $|V(G_n)|=n$, $|E(G_n)|\ge cn^{1+1/k}$ and the girth of $G_n$ is more than $2k$.
\end{conjecture}
    
This conjecture has been verified for $k=2,3,5$, see \cite{B1966,Br1966,R1958,W1991}. For a general $k$, Sudakov and Verstra\"ete \cite{SV2008} showed that if such graphs exist, then they contain a $C_{2l}$ for any $l$ with $k < l \le Cn$, for some constant $C > 0$. More recently, Solymosi and Wong \cite{SW2017} proved that if such graphs exist, then in fact, they contain many $C_{2l}$'s  for any fixed $l > k$. More precisely they proved:

\begin{thm}[Solymosi, Wong \cite{SW2017}]\label{solymosiwong} If Erd\H os's Girth Conjecture holds for $k$, then for every $l > k$ we have $$ex(n,C_{2l},\cC_{2k})=\Omega(n^{2l/k}).$$
\end{thm}

The following remark shows that in many cases this bound is sharp.
\begin{remark}
\label{SolymosiWong_upper}
If $k+1$ divides $2l$, then $$ex(n,C_{2l},\cC_{2k})=O(n^{2l/k}).$$ Indeed, let us associate to each $C_{2l}$, one fixed ordered list of $2l/(k+1)$ edges $(e_1, e_{k+1}, e_{2k+1}, \ldots)$, where $e_1$ appears as the first edge (chosen arbitrarily) on the $C_{2l}$, $e_{k+1}$ as the $(k+1)$-th edge, $e_{2k+1}$ as the $(2k+1)$-th edge and so on. Note that at most one $C_{2l}$ is associated to an ordered tuple $(e_1, e_{k+1}, e_{2k+1}, \ldots)$, because there is at most one path of length $k-1$ connecting the endpoints of any two edges (as all the short cycles are forbidden).  Since there are at most $O(n^{1+1/k})$ ways to select each edge, this shows the number of $C_{2l}$'s is at most $O((n^{1+1/k})^{2l/(k+1)}) = O(n^{2l/k})$, showing that the bound in Theorem \ref{solymosiwong} is sharp when $k+1$ divides $2l$.

\end{remark}

It is worth mentioning that Gerbner, Keszegh, Palmer and Patk\'os \cite{GKPP2016} considered a similar problem, where a finite list of allowed cycle lengths is given (thus the list of forbidden cycle lengths is infinite). Another main difference is that in \cite{GKPP2016}, all cycles of allowed lengths are counted, as opposed to only counting the number of cycles of a given length like in this paper.

\subsection*{Constructions}
Before mentioning our results in the next section, we present typical constructions of graphs (with many copies of a cycle) that we will refer to, in the rest of the paper.
\begin{itemize}
	\item  For $l,t \ge 1$ the \textit{(l, t)-theta-graph} with endpoints $x$ and $y$ is the graph obtained by joining two vertices $x$ and $y$, by $t$ internally disjoint paths of length $l$.
	
	\item For a simple graph $F$ and $n,l \ge 1$ the \textit{theta-$(n,F,l)$ graph} is a graph on $n$ vertices obtained by replacing every edge $xy$ of $F$ by an $(l, t)$-theta-graph with endpoints $x$ and $y$, where $t$ is chosen as large as possible, with some isolated vertices if needed. 
	More precisely let $t=\lfloor\frac{n-|V(F)|}{|E(F)|(l-1)}\rfloor$, and we add $n-(t|E(F)|(l-1)+|V(F)|)$ isolated vertices.
	
\end{itemize}

\section{Our results}

\subsection{Forbidding a cycle of given length}

We determine the order of magnitude of $ex(n,C_{2l},C_{2k})$ below.

	\begin{thm}
		\label{main2} For any $l\ge 3$ and $k \ge 2$ we have
        $$ex(n,C_{2l},C_{2k})\le (1+o(1)) \frac{2^{l-2} (k-1)^l}{2l}n^l.$$
		For any $k > l \ge 2$ we have
		$$ex(n,C_{2l},C_{2k})\ge  (1+o(1))\frac{(k-1)_l}{2l}n^l.$$
		For any $l > k \ge 3$ we have $$ex(n,C_{2l},C_{2k})\ge  (1+o(1))\frac{1}{l^l}n^l.$$
	\end{thm}

Theorem \ref{main2} and Theorem \ref{four_cycle} (stated below) show that $ex(n,C_{2l},C_{2k}) = \Theta(n^l)$ for any $k, l \ge 2$, except for the lower bound in the case $k =2$, which can be easily shown by counting cycles in the well-known $C_4$-free graph constructed by Erd\H{o}s and R\'{e}nyi \cite{ER} (See Theorem \ref{celebrated} and \cite{gs2017}).

This theorem has recently been proven independently by Gishboliner and Shapira \cite{gs2017}. Our proof is different from theirs, and it gives a better bound if $k$ is fixed (moreover, if $l$ is fixed, then their bound and our bound are both tight). They study odd cycles as well, determining the order of magnitude of $ex(n,C_l,C_k)$ for every $l>3$ and $k$, and also provide interesting applications of these results in the study of the graph removal lemma and graph property testing.

Solymosi and Wong \cite{SW2017}  asked whether a similar lower bound (to that of Theorem \ref{solymosiwong}) on the number of $C_{2l}$'s holds, if just $C_{2k}$ is forbidden instead of forbidding $\cC_{2k}$. Theorem \ref{main2} answers this question in the negative.

\subsubsection*{Asymptotic results}

If we go beyond determining the order of magnitude we can ask the asymptotics of these functions. In many cases it is a much harder question than the order of magnitude question \cite{ALS2016,BGy2008,C5C3,MQ2018}.

\vspace{2mm}

We determine $ex(n,C_4,C_{2k})$ asymptotically.

\begin{thm}\label{four_cycle} For $k \ge 2$ we have
		$$ex(n,C_4,C_{2k})= (1+o(1)) \frac{(k-1)(k-2)}{4} n^2.$$
	\end{thm}

Since most constructions are bipartite, it is natural to consider the bipartite version of the generalized Tur\'an function: Let $ex_{bip}(n, C_{2l}, C_{2k})$ denote the maximum number of copies of a $C_{2l}$ in a bipartite $C_{2k}$-free graph on $n$ vertices. Our methods give sharper bounds for $ex_{bip}(n, C_{2l}, C_{2k})$ compared to the bounds in Theorem \ref{main2} (see Remark \ref{bipartite_version}). In the case $l = 3, \   k =4$ we can determine the asymptotics.

\begin{thm}\label{bipC6C8}
	We have	$$ex_{bip}(n,C_6,C_8)=n^{3}+O(n^{5/2}).$$
	\end{thm}
 

We prove this theorem in Subsection \ref{C_6C_8}. Interestingly, the proof makes use of Corollary \ref{gyorilemons} concerning Berge-$C_4$-free hypergraphs. We leave open the question of determining the asymptotics of $ex(n,C_6,C_8)$, which we believe to be the same as that of $ex_{bip}(n,C_6,C_8)$.

\subsection{Forbidding a set of cycles}

Theorem \ref{solymosiwong} implies that if Erd\H os's Girth Conjecture is true (recall that it is known to be true for $l=2,3,5$), then $ex(n,C_{2l},\cC_{2l-2})=\Omega(n^{2l/(l-1)})$ for any $l \ge 3$. On the other hand, by Remark \ref{SolymosiWong_upper}, this number is at most $O(n^{2l/(l-1)})$. This implies $ex(n,C_{2l},\cC_{2l-2})=\Theta(n^{2l/(l-1)}).$ By Lemma \ref{paros} (which is straightforward to prove), we know that when counting copies of an even cycle, forbidding an odd cycle does not change the order of magnitude. Therefore, we have

\begin{corollary}
\label{evengirth}
Suppose $l \ge 3$ and Erd\H os's Girth Conjecture is true for $l-1$. Then we have $$ex(n,C_{2l},\cC_{2l-1})=\Theta(n^{2l/(l-1)}).$$
\end{corollary}

In other words, the maximum number of $C_{2l}$'s in a graph of girth $2l$ is $\Theta(n^{2l/(l-1)})$. We prove that the previous theorem is sharp in the sense that forbidding one more even cycle decreases the order of magnitude significantly: The maximum number of $C_{2l}$'s in a $C_{2k}$-free graph with girth $2l$ is $\Theta(n^2).$ That is, 
$$ex(n,C_{2l},\cC_{2l-1} \cup \{C_{2k}\})=\Theta(n^2).$$

More generally, we show the following.

\begin{thm}\label{longer_kor} For any $k > l$ and $m\ge 2$ such that $2k \neq ml$ we have $$ex(n,C_{ml},\cC_{2l-1} \cup \{C_{2k}\})=\Theta(n^m).$$
    \end{thm}


	
Observe that forbidding even more cycles does not decrease the order of magnitude, as long as we do not forbid $C_{2l}$ itself, as shown by $(l,\lfloor n/l \rfloor)$-theta graph and some isolated vertices (i.e. the theta-$(n,K_2,l)$ graph). On the other hand it is easy to see that if we forbid every cycle of length other than $2l+1$, then there are $O(n)$ copies of $C_{2l+1}$. 

Corollary \ref{evengirth} determines the order of magnitude of maximum number of $C_{2l}$'s in a graph of girth $2l$. It is then very natural to consider the analogous question for odd cycles: What is the maximum number of $C_{2k+1}$'s in a graph of girth $2k+1$? Before answering this question, we state a strong form of Erd\H{o}s's Girth Conjecture that is known to be true for small values of $k$.

A graph $G$ on $n$ vertices, with average degree $d$, is called \textit{almost-regular} if the degree of every vertex of $G$ is $d+ O(1)$. 
\begin{conjecture}[Strong form of Erd\H{o}s's Girth Conjecture]
\label{strongerEGC}
 For any positive integer $k$, there exist a family of almost-regular graphs $\{G_n\}$ such that $|V(G_n)|=n$, $|E(G_n)|\ge \frac{n^{1+1/k}}{2}$ and $G_n$ is $\{C_4,C_6, \ldots, C_{2k}\}$-free.
\end{conjecture}
Lazebnik, Ustimenko and Woldar \cite{Lazebnik_U_Woldar} showed Conjecture \ref{strongerEGC} is true when $k \in \{2, 3, 5\}$ using the existence of polarities of generalized polygons. We show the following that can be seen as the `odd cycle analogue' of Theorem \ref{solymosiwong}. 

\begin{thm}
\label{odd_girth}
Suppose $k \ge 2$ and Strong form of Erd\H{o}s's Girth Conjecture is true for $k$. Then we have $$ex(n,C_{2k+1},\cC_{2k})= (1+o(1))\frac{n^{2+\frac{1}{k}}}{4k+2}.$$
\end{thm}
 
To show that Theorem \ref{odd_girth} is sharp and to give an analogue of Theorem \ref{longer_kor} (in the case of $m=2$) for odd cycles, we prove that if we forbid one more even cycle, then the order of magnitude goes down significantly:

\begin{thm}\label{oddgirthpluseven}

For any integers $k > l \ge 2$, we have
$$\Omega(n^{1+\frac{1}{2k+1}}) = ex(n, C_{2l+1}, \cC_{2l} \cup \{C_{2k}\}) = O(n^{1+\frac{l}{l+1}}).$$

\end{thm}

However, if the additional forbidden cycle is of odd length, we can only prove a quadratic upper bound. We conjecture that the truth is also sub-quadratic here (see Section \ref{concludingremarks}, Theorem \ref{supporting_conjecture}).

\begin{thm}\label{oddgirthplusodd} For any integers $k > l \ge 2$, we have
$$\Omega(n^{1+\frac{1}{2k+2}}) = ex(n, C_{2l+1}, \cC_{2l} \cup \{C_{2k+1}\}) = O(n^2).$$ 
\end{thm}

Concerning forbidding a set of cycles we also determine the asymptotics of $ex(n,C_4,\cC_A)$ for every possible set $A$. Let $A_e$ be the set of even numbers in $A$ and $A_o$ be the set of odd numbers in $A$.

\begin{thm}\label{morecycasy1} For any $k \ge 3$, we have
    $$ex(n,C_4,\cC_A)=\left\{ \begin{array}{l l}
		0 & \textrm{if $4\in A$}\\
		(1+o(1)) \frac{(k-1)(k-2)}{4} n^2 & \textrm{if $4\not\in A$ and $2k$ is the smallest element of $A_e$}\\
		(1+o(1)) \frac{1}{64} n^4  & \textrm{if $A_e=\emptyset$.}
		\end{array}
		\right.$$
	\end{thm}

We also determine the order of magnitude of $ex(n,C_6,\cC_A)$ by proving

\begin{thm}\label{morecycasy2} $$ex(n,C_6,\cC_A)=\left\{ \begin{array}{l l}
		0 & \textrm{if $6\in A$,}\\
		\Theta(n^2) & \textrm{if $6\not\in A$, $4\in A$ and $|A_e|\ge 2$,}\\
		\Theta(n^3) & \textrm{if $4,6\not\in A$ and $A_e\neq \emptyset$, or if $A_e=\{4\},$}\\
		\Theta(n^6) & \textrm{if $A_e=\emptyset$.}
		\end{array}
		\right.$$
\end{thm}        
  
\subsection{Maximum number of $P_l$'s in a graph avoiding a cycle of given length}



We study the maximum possible number of paths in a $C_{2k}$-free graph and prove the following results.

\begin{thm}\label{pathevenupper}
For $l, k \ge 2$, we have $$ex(n, P_l, C_{2k}) \le (1+o(1)) \frac{1}{2} (k-1)^{\frac{l-1}{2}}n^{\frac{l+1}{2}} .$$
\end{thm}

\begin{thm}\label{pathevenlower}
If $2 \le l < 2k$, then $$ex(n, P_l, C_{2k}) \ge  (1+o(1))\frac{1}{2} (k-1)_{\lfloor \frac{l}{2} \rfloor} n^{\lceil \frac{l}{2} \rceil}.$$
If $l \ge 2k$, then $$ex(n, P_l, C_{2k}) \ge (1+o(1)) \max \left \{\left(\frac{n}{\lfloor{l/2}\rfloor}\right)^{\lceil l/2 \rceil}, \left(\frac{(k-1)}{4(k-2)^{k+2}}\right)^{\lceil \frac{l}{2} \rceil} (k-1)_{\lfloor \frac{l}{2} \rfloor}n^{\lceil \frac{l}{2} \rceil} \right \}.$$
\end{thm}

Note that if $l < 2k$ and $l$ is odd, Theorem \ref{pathevenupper} and Theorem \ref{pathevenlower} show that $ex(n, P_l, C_{2k})$ is equal to $(1+o(1)) \frac{1}{2} (k-1)^{\frac{l-1}{2}}n^{\frac{l+1}{2}}$ as $k$ and $n$ tend to infinity. 

Finally, we determine the maximum number of copies of $P_l$ in a $C_{2k+1}$-free graph, asymptotically.

\begin{thm}\label{pathodd}
For $k \ge 1$ and $l \ge 2$, we have $$ex(n, P_l, C_{2k+1}) = (1+o(1)) \left(\frac{n}{2}\right)^{l}.$$
\end{thm}

\vspace{1cm}

    \textbf{Structure of the paper:} In Section 3 we prove Theorem \ref{main2}, determining the order of magnitude of $ex(n, C_{2l}, C_{2k})$
    
    In Section 4 we determine the asymptotics of $ex(n,C_4,C_{2k})$ and $ex_{bip}(n, C_6, C_8)$ (Theorem \ref{four_cycle} and Theorem \ref{bipC6C8}).
    
    In Section 5 we prove some basic lemmas for the case when a set of cycles are forbidden and prove Theorem \ref{longer_kor} concerning graphs of even girth,  along with results about $ex(n,C_4,\cC_A)$ and $ex(n,C_6,\cC_A)$ for every possible set $A$ (i.e., Theorem \ref{morecycasy1} and Theorem \ref{morecycasy2}). 
    
    In Section 6 we prove the theorems concerning graphs of odd girth: Theorem \ref{odd_girth}, Theorem \ref{oddgirthpluseven} and Theorem \ref{oddgirthplusodd}. 
    
    In Section 7 we count number of copies of $P_l$ in a $C_k$-free graph and prove Theorem \ref{pathevenupper}, Theorem \ref{pathevenlower} and Theorem \ref{pathodd}. 
    
    Finally in Section 8, we make some remarks and pose questions.

    	\section{Maximum number of $C_{2l}$'s in a $C_{2k}$-free graph}
	
	Below we prove Theorem \ref{main2}. Note that the case $l=2$ of Theorem \ref{longer_kor} gives back Theorem \ref{main2}, and the proof here is also a special case of the proof of that more general statement. We decided to include it here separately for two reasons. On the one hand, this is an important special case. On the other hand, it serves as an introduction to the similar, but more complicated proof of Theorem \ref{longer_kor}.
    
	\begin{proof}[Proof of Theorem \ref{main2}]
		
		Let us start with the lower bound and assume first that $2 \le l<k$. Then $K_{k-1,n-k+1}$ is $C_{2k}$-free and  it contains $$\frac{1}{2l}\binom{k-1}{l}\binom{n-k+1}{l}l!l!= (1+o(1))\frac{(k-1)(k-2)\dots (k-l)}{2l}n^l = (1+o(1))\frac{(k-1)_l}{2l} n^l$$ copies of $C_{2l}$.

        
		Let us now assume $l>k \ge 3$. Consider a copy of $C_{2l}$ and replace every second vertex $u$ by $\lfloor n/l-1 \rfloor$ or $\lceil n/l-1 \rceil$ copies of it, each connected to the two neighbors of $u$ in the $C_{2l}$. The resulting graph only contains cycles of length $4$ and $2l$, thus it is $C_{2k}$-free and it contains $$(1+o(1)) \frac{1}{l^l}n^l$$ copies of $C_{2l}$.
		
		\
		Let us continue with the upper bound. Consider a $C_{2k}$-free graph $G$. First we introduce the following notation. For two distinct vertices $a, b \in V(G)$, let $$f(a,b):= \textrm{ number of common neighbors of } a \textrm{ and } b.$$ Then we have \begin{equation}\label{equa1}\frac{1}{2}\sum_{a\neq b,\,a,b \in V(G)}\binom{f(a,b)}{2}\le (1+o(1)) \frac{(k-1)(k-2)}{4}n^2 ,\end{equation} by Theorem \ref{four_cycle}, since the left-hand-side is equal to the number of $C_4$'s in $G$.
		

		\begin{claim}\label{klem} For every $a\in V(G)$ we have $$\sum_{b \in V(G) \setminus \{a\}} f(a,b)\le (2k-2)n.$$
			
		\end{claim}
		Note that the left-hand side of the above inequality is the number of $P_3$'s starting at $a$.
		
		\begin{proof} 
			Recall that $N_1(a)$ is the set of vertices adjacent to $a$ and $N_2(a)$ is the set of vertices at distance exactly $2$ from $a$.
			Let $E_1$ be the set of edges induced by $N_1(a)$ and $E_2$ be the set of edges $uv$ with  $u \in N_1(a)$ and $v \in N_2(a)$. It is easy to see that $\sum_{b\in V(G) \setminus \{a\}}f(a,b)=2|E_1|+|E_2|$.
			
			We claim that there is no cycle of length longer that $2k-2$ in $E_1 \cup E_2$. 
			
			First suppose by contradiction that there is a cycle $C$ of length $2k-1$ in $E_1 \cup E_2$. Since the cycle is of odd length it must contain an edge $uv \in E_1$. The subpath of length $2k-2$ between the vertices $u$ and $v$ in $C$ together with the edges $ua$ and $va$ forms a $C_{2k}$ in $G$, a contradiction. 
			
			Now suppose that there is a cycle $C$ of length at least $2k$ in $E_1 \cup E_2$. Observe first that a subpath of length $2k-2$ of $C$ starting from a vertex in $N_1(a)$ cannot have its other endpoint in $N_1(a)$, as that would form a $C_{2k}$ together with the vertex $a$. Thus there has to be an edge $v_1v_2$ of $E_1$ in $C$. Consider the subpath $v_1, v_2, \dots, v_{2k-1} v_{2k}$ of $C$. The vertices $v_{2k-1}$ and $v_{2k}$ are both in $N_2(a)$ because they are endpoints of paths of length $2k-2$ starting in $v_1$ and $v_2$ respectively. But then, the edge $v_{2k-1}v_{2k} \in E(C)$ is not in $E_1 \cup E_2$, a contradiction again.
			
			Then by Theorem \ref{erga} we have $|E_1| + |E_2| = |E_1 \cup E_2| \ \le (k-1)n$, which implies the claim. 

		\end{proof}
		
		The above claim implies that we have 
		\begin{equation}\label{equa2}\sum_{a\neq b,\,a,b\in V(G)}f(a,b)\le (k-1)n^2.
		\end{equation}
		
		Let us fix vertices $v_1, v_2, \dots, v_l$ and let $g(v_1,v_2, \ldots, v_l)$ be the number of $C_{2l}$'s in $G$ where $v_i$ is the $2i$-th vertex ($i\le l$). 
		 Clearly $g(v_1,v_2, \ldots, v_l) \le \prod_{j=1}^l f(v_j,v_{j+1})$ (where $v_{l+1}=v_1$ in the product).
%
		 If we add up $g(v_1,v_2, \ldots, v_l)$ for all possible $l$-tuples $(v_1, v_2, \dots, v_l)$ of $l$ distinct vertices in $V(G)$, we count every $C_{2l}$ exactly $4l$ times. It means the number of $C_{2l}$'s is at most
		
		\begin{equation}\label{equa3}\frac{1}{4l}\sum_{(v_1, v_2, \dots, v_l)}\prod_{j=1}^l f(v_jv_{j+1})\le\frac{1}{4l}\sum_{(v_1, v_2, \dots, v_l)}\frac{f^2(v_1,v_2)+f^2(v_2,v_3)}{2}\prod_{j=3}^l f(v_jv_{j+1}).
		\end{equation}
		
		Fix two vertices $u,v\in V(G)$ and let us examine what factor $f^2(u,v)$ is multiplied with in \eqref{equa3}. It is easy to see that $f^2(u,v)$ appears in \eqref{equa3} whenever $u=v_1,v=v_2$ or $u=v_2,v=v_1$ or $u=v_2,v=v_3$ or $u=v_3,v=v_2$.
		Let us start with the case $u=v_1$ and $v=v_2$. 
		Then $f^2(u,v)$ is multiplied with $$\frac{1}{8l}\left(\prod_{j=3}^{l-1} f(v_jv_{j+1})\right)f(v_l,u) = \frac{1}{8l}f(u,v_l) \prod_{j=3}^{l-1} f(v_jv_{j+1})$$ for all the choices of tuples $(v_3, v_4, \dots, v_l)$ (where these are all different vertices). We claim that $$\sum_{(v_3, v_4, \dots, v_l)}\frac{1}{8l}f(u,v_l)\prod_{j=3}^{l-1} f(v_jv_{j+1})\le \frac{(2k-2)^{l-2}n^{l-2}}{8l}.$$ Indeed, we can rewrite the left-hand side as $$\frac{1}{8l}\sum_{v_l\in V(G)}f(u,v_l)\sum_{v_{l-1}\in V(G)}f(v_l,v_{l-1})\dots \sum_{v_{3}\in V(G)}f(v_4,v_{3}),$$ and each factor is at most $(2k-2)n$ by Claim \ref{klem}.

		
		Similar calculation in the other three cases gives the same upper bound, so adding up all four cases, we get an additional factor of 4, showing that the number of copies of $C_{2l}$ is at most
		$$\frac{1}{2l}\sum_{u \not = v, u,v\in V}f^2(u,v)(2k-2)^{l-2}n^{l-2}.$$
		
		Finally observe that, $$\sum_{u \not = v, u,v\in V}f^2(u,v)=2\sum_{u \not = v, u,v\in V}\binom{f(u,v)}{2}+\sum_{u \not = v, u,v\in V}f(u,v)\le (1+o(1)) (k-1)(k-2)n^2 +(k-1)n^2,$$
		
which is at most $(1+o(1))(k-1)^2n^2$. Note that the above inequality follows from (\ref{equa1}) and (\ref{equa2}).  This finishes the proof.
	\end{proof}
	
	\begin{remark}
    \label{bipartite_version}
    Note that if $G$ is bipartite, or even just triangle-free, then in Claim \ref{klem}, $E_1$ is empty. Therefore the same proof gives the better upper bound $\sum_{b \in V(G) \setminus \{a\}} f(a,b)\le (k-1)n.$ So we get $$\sum_{a\neq b,\,a,b\in V(G)}f(a,b)\le \frac{k-1}{2} n^2$$ instead of \eqref{equ1}. Hence we can obtain 
    $$ ex_{bip}(n,C_{2l},C_{2k})\le ex(n,C_{2l},\{C_3,C_{2k}\})\le (1+o(1)) \frac{(k-\frac{3}{2})(k-1)^{l-1}}{2l}n^l.$$
    Observe that the construction given in Theorem \ref{main2} is bipartite. Thus the ratio of the upper and lower bounds of $ex_{bip}(n,C_{2l},C_{2k})$ is $$\frac{(k-\frac{3}{2})(k-1)^{l-2}}{(k-1)_l},$$ which goes to 1 as $k$ increases.
    
		
		
		
		
	\end{remark}

	\section{Asymptotic results}
    

We will first prove the following simple result and use it in the proof of Theorem \ref{morecycasy1}.

	\begin{thm}\label{evenodd} For any $k, l$, we have $$ex(n,C_{2l},C_{2k+1}) = (1+o(1))\frac{1}{2l}\left    (\frac{n^2}{4}\right   )^{l}.$$
		
	\end{thm}
	
	\begin{proof} 
    The lower bound is given by the complete bipartite graph $K_{n/2, n/2}$.
    
    Let $G$ be a graph which is $C_{2k+1}$-free. By a theorem of Gy\H ori and Li \cite{GyoriLi} there are at most $O(n^{1+1/k})$ triangles in $G$, so let us delete an edge from each of them and call the resulting triangle-free graph $G'$. This way we delete at most $O(n^{1+1/k}) n^{2l-2} = o(n^{2l})$ copies of $C_{2l}$. So it suffices to estimate the number of $C_{2l}$'s in $G'$. 
    
   First we count the number of ordered tuples of $l$ independent edges $M_l = (e_1, e_2, \ldots, e_l)$ in $G'$. As the maximum number of edges in a triangle-free graph is at most $\lfloor n^2/4\rfloor$ by Mantel's theorem, we can pick an edge $e_1 = uv$ of $G$ in at most $\lfloor n^2/4\rfloor$ ways. Then we can pick the edge $e_2$ disjoint from $e_1$ in at most $\lfloor (n-2)^2/4\rfloor$ ways as the subgraph of $G$ induced by $V(G) \setminus \{u, v\}$ is also triangle-free. We can pick $e_3$ in at most $\lfloor (n-4)^2/4\rfloor$ ways, $e_4$ at most $\lfloor (n-6)^2/4\rfloor$ ways and so on. Thus $G'$ contains at most $$\Big \lfloor \frac{n^2}{4} \Big \rfloor \Big \lfloor \frac{(n-2)^2}{4} \Big \rfloor \Big \lfloor \frac{(n-4)^2}{4} \Big \rfloor\dots \Big \lfloor \frac{(n-2l+2)^2}{4} \Big \rfloor = (1+o(1))\left(\frac{n^2}{4}\right)^{l}$$ copies of $M_{l}$.
   
   Now we count the number of $C_{2l}$'s containing a fixed copy of $M_l = (e_1, e_2, \ldots, e_l)$, where $e_i = u_iv_i$. To obtain a cycle $C_{2l}$ from $M_l$, we decide for every $i$, whether $u_i$ follows $v_i$ or $v_i$ follows $u_i$ in a clock-wise ordering. However, for any given $i$, after deciding the order for $u_i$ and $v_i$, we claim that the order for $u_{i+1}$, $v_{i+1}$ is determined. Indeed, suppose w.l.o.g  that $v_i$ follows $u_i$. Then $v_i$ can be adjacent to at most one of the vertices $u_{i+1}$, $v_{i+1}$ because $G'$ is triangle-free. Thus the order of $u_{i+1}$, $v_{i+1}$ is determined. So once  the order of $u_1$ and $v_1$ is fixed (in two ways) the cycle $C_{2l}$ is determined. Thus the number of $C_{2l}$'s obtained from a fixed copy of $M_l$ is at most $2$. Note that in this way each copy of $C_{2l}$ in $G'$ is obtained exactly $4l$ times. So the total number of $C_{2l}$'s in $G'$ is at most $$(1+o(1))\left(\frac{n^2}{4}\right)^{l} \cdot \frac{2}{4l} =  (1+o(1))\frac{1}{2l}\left(\frac{n^2}{4}\right)^{l}$$ as required.
\end{proof}

\subsection{Proof of Theorem \ref{four_cycle}: Maximum number of $C_4$'s in a $C_{2k}$-free graph}
  
  We restate Theorem \ref{four_cycle} below for convenience.
  
	\begin{thm*}\label{mult1}
		For $k \ge 2$ we have:
		$$ex(n,C_4,C_{2k})= (1+o(1)) \frac{(k-1)(k-2)}{4} n^2.$$
	\end{thm*}
	
	\begin{proof} For the lower bound consider the complete bipartite graph $K_{k-1,n-k+1}$.
		
		\smallskip
		
		For the upper bound consider a $C_{2k}$-free graph $G$.
		We call a pair of vertices \emph{fat} if they have at least $k$ common neighbors, otherwise it is called non-fat. We call a $C_4$ \textit{fat} if both pairs of opposite vertices in that $C_4$ are fat. First we claim that the number of non-fat $C_4$'s is at most $\binom{k-1}{2}\binom{n}{2}$. Indeed, there are at most $\binom{n}{2}$ non-fat pairs, and each of them is contained in at most $\binom{k-1}{2}$ $C_4$'s as an opposite pair.
		
		In the remaining part of the proof we will prove that the number of fat $C_4$'s is $O(n^{1+1/k})$, by using an argument inspired by the reduction lemma of Gy\H ori and Lemons \cite{GyL2012}. We go through the fat $C_4$'s in an arbitrary order, one by one, and pick exactly one edge (from the four edges of the $C_4$); we always pick the edge which was picked the smallest number of times before (in case there is more than one such edge, then we pick one of them arbitrarily). 
		
		After this procedure, every edge $e$ is picked a certain number of times. Let us denote this number by $m(e)$, and we call it the \textit{multiplicity} of $e$. Note that $\sum_{e \in E(G)} m(e)$ is equal to the number of fat $C_4$'s in $G$. 
		
	If $$m(e) < 2(k-2)k^2\binom{2k}{k}$$ for each edge $e$, then the number of fat $C_4$'s in $G$ is at most $$2(k-2)k^2\binom{2k}{k} \abs{E(G)} = O(n^{1+1/k})$$ by Theorem \ref{bs}, as desired.
		
%
		
		Hence we can assume there is an edge $e$ with $m(e) \ge 2(k-2)k^2\binom{2k}{k}$. In this case we will find a $C_{2k}$ in $G$, which will lead to a contradiction, finishing the proof. More precisely, we are going to prove the following statement: 
		
		\begin{claim}
		For every $2 \le l \le k$ there is a $C_{2l}$ in $G$, that contains an edge $e_l$ with $$m(e_l)\ge 2(k-l)k^2\binom{2k}{k}.$$ 
		\end{claim}
		\begin{proof}
		We prove it by induction on $l$. For the base case $l = 2$, consider any $C_4$ containing $e = e_2$. Let us assume now we have found a cycle $C$ of length $2l$ in $G$ and one of its edges $e_l=uv$ has $$m(e_l)\ge 2(k-l)k^2\binom{2k}{k}.$$ 
		
		
		For any $i \le 2(k-l)k^2\binom{2k}{k}$, when $e_l$ was picked for the $i$th time, the corresponding fat $C_4$ contained four edges each of which had been picked earlier at least $i-1$ times, thus they have multiplicity at least $i-1$. Let $\cF_l$ be the set of those fat $C_4$'s where $e_l$ was picked for the last $2k^2\binom{2k}{k}-1$ times. All the three other edges of each of these fat $C_4$'s have multiplicity at least $$2(k-l)k^2\binom{2k}{k}-2k^2\binom{2k}{k} =     2(k-l-1)k^2\binom{2k}{k}.$$ 
		
	   At most $\binom{2l-2}{2}$ of the $C_4$'s in $\cF_l$ have all four of their vertices in $C$ (note that they all contain the edge $e_l$). 
	   
	   Observe that $G$ is $K_{k,k}$-free, as $C_{2k}$ is a subgraph of $K_{k,k}$. This means that any $k$ vertices in $C$ have at most $k-1$ common neighbors. We claim that there are at most $(k-1)\binom{2l}{k}$ vertices  in $V(G) \setminus V(C)$ that are connected to at least $k$ vertices in $C$.  Indeed, otherwise by pigeon hole principle, there are $k$ vertices in $C$ such that each of them is connected to the same $k$ vertices in $V(G) \setminus V(C)$, a contradiction. Therefore, at most $(2l-2)(k-1)\binom{2l}{k}$ $C_4$'s have a vertex in $C$ and a vertex $w$ outside $C$ such that $w$ is connected to at least $k$ vertices in $C$.
		
		Thus, there are at least $$(2k^2\binom{2k}{k}-1) - \binom{2l-2}{2} -(2l-2)(k-1)\binom{2l}{k} \ge 1$$ four-cycle(s) in $\cF_l$ such that one of the following two cases hold. Let $uvxyu$ be one such four-cycle (recall that $e_l = uv$).

\ 

		\textbf{Case 1.} $x, y \in V(G) \setminus V(C).$ 
		
		We replace the edge $e_l$ of $C$ by the path consisting of the edges $vx, xy, yu$, thus obtaining a cycle of length $2l+2$. The edges $vx, xy, yu$ have multiplicity at least $2(k-l-1)k^2\binom{2k}{k}$, which finishes the proof in this case.

\ 

		\textbf{Case 2.} $x \in V(C), y \not \in V(C)$ and $y$ has less than $k$ neighbors in $C$.  
		
		Note that in this case $\{y,v\}$ is a fat pair, thus $y$ and $v$ have at least $k$ common neighbors. At least one of those, say $w$, is not in $C$. Let us replace the edge $e_l$ of $C$ by the path consisting of the edges  $uy$, $yw$, $wv$. This way we obtain a cycle of length $2l+2$, and one of its edges $uy$ has $m(uy)\ge 2(k-l-1)k^2\binom{2k}{k}$, which finishes the proof of the claim and the theorem.
	\end{proof}
	
\end{proof}

\subsection{Proof of Theorem \ref{bipC6C8}: Maximum number of $C_6$'s in a bipartite $C_8$-free graph}
\label{C_6C_8}

Let us recall that Theorem \ref{bipC6C8} states $ex_{bip}(n,C_6,C_8)=n^{3}+O(n^{5/2})$.

Let $G$ be a $C_8$-free bipartite graph with
classes $A$ and $B$. 
Let us define a pair of vertices $u,v$ from the same class \textit{fat}, if 
there are four different vertices $w_1,w_2,w_3,w_4$ from the other class such that $uw_1,uw_2,uw_3,uw_4,vw_1,vw_2,vw_3$ and $vw_4$ are edges of the graph. 

\begin{claim}\label{fatfat1}

Suppose $v_1v_2v_3v_4v_5v_6v_1$ is a $6$-cycle and $v_1,v_3$ is a fat pair. Then neither $v_2,v_4$ nor $v_2,v_6$ is a fat pair.

\end{claim}

\begin{proof} We prove it by contradiction. Let $w$ be a neighbor of both $v_1$ and $v_3$ with $w \not \in \{v_2,v_4,v_6\}$ and let $u$ be a be a neighbor of both vertices in the other fat pair (either $v_2$ and $v_4$, or $v_2$ and $v_6$) with $u \not \in \{v_1,v_3,v_5\}$. We can find such $w$ and $u$ because of the definition of fatness. Then $v_1wv_3v_2uv_4v_5v_6v_1$ or $v_1v_2uv_6v_5v_4v_3wv_1$ is a $C_8$, a contradiction.

\end{proof}

So by Claim \ref{fatfat1} we can suppose that if $v_1v_2v_3v_4v_5v_6v_1$ is a 6-cycle and $v_1,v_3$ is a fat pair, then there can be only one fat pair among $v_2, v_4, v_6$ and it is $v_4,v_6$. 
This means that if there are two fat pairs among the vertices of a $6$-cycle in different classes then (up to permutation of vertices) they should be $v_1,v_3$ and $v_4,v_6$. 
Let us call a $6$-cycle \emph{fat} if it contains fat edges from both classes $A$ and $B$. First we are going to prove that there are $O(n^{2.5})$ fat $6$-cycles in $G$. 

Let $v_1v_2v_3v_4v_5v_6v_1$ and $v_1v'_2v_3v'_4v'_5v'_6v_1$ be two different fat $6$-cycles (i.e., their fat pairs coincide in one of the classes).  Claim \ref{fatfat1} implies that $v_4,v_6$ and $v'_4,v'_6$ are the other fat pairs in these cycles. 

\begin{claim}\label{v_1v_3}

We have $\{v_4,v_6\} \cap \{v'_4,v'_6\} \neq \emptyset.$

\end{claim}

\begin{proof}

Let us suppose by contradiction that $\{v_4,v_6\} \cap \{v'_4,v'_6\} = \emptyset.$ Then by the definition of fatness we can choose $u \not \in \{v_1,v_3\}$ and $w \not \in \{v_1,v_3,u\}$ such that $u$ is connected to both $v_4$ and $v_6$, and $w$ is connected to both $v'_4$ and $v'_6$. Then $v_1v_6uv_4v_3v'_4wv'_6v_1$ is a $C_8$, a contradiction. 
\end{proof}

Let $N(u,v):=\{w : uw,vw \in E(G) \}.$
We prove that for distinct fat pairs these sets of common neighborhoods are almost disjoint.

\begin{claim}\label{linear} Let $v_1,v_3,v'_1,v'_3 \in A$ with $\{v_1,v_3\} \neq \{v'_1,v'_3\}$ such that $\{v_1,v_3\}$ and $\{v'_1,v'_3\}$ are fat pairs of two fat $6$-cycles. Then we have

$$|N(v_1,v_3) \cap N(v'_1,v'_3)| \ \le 1.$$

\end{claim}

\begin{proof}

We prove it by contradiction, let us suppose that we have different vertices $x,y \in N(v_1,v_3) \cap N(v'_1,v'_3)$. By our assumption we can suppose that $v'_1 \not \in \{v_1,v_3\}$. Let $v_1v_2v_3v_4v_5v_6v_1$ be a fat cycle.

Assume first $\{x,y\} \cap \{v_4,v_6\}=\emptyset$. As the pair $v_4,v_6$ is fat, we can find $u \not \in \{v'_1,v_1,v_3\}$ that is connected to both $v_4$ and $v_6$. Then $xv'_1yv_3v_4uv_6v_1$ is a $C_8$, a contradiction.

Hence we can assume $x=v_4$. Consider now the case $v_6\neq y$. By the fatness of $v_1,v_3$ there is a $u\not\in \{x,y,v_6\}$ connected to both $v_1$ and $v_3$. By the fatness of $v_4,v_6$ there is $w\not\in\{v_1,v_3,v'_1\}$ connected to both $v_4$ and $v_6$. Then $v_1uv_3yv'_1v_4wv_6v_1$ is a $C_8$, a contradiction.

Thus we have $x=v_4$, $y=v_6$. Assume first $v'_3\not\in \{v_1,v_3\}$. By the fatness of $v'_1,v'_3$ we have $w\not\in \{x,y,v_2\}$ connected to both $v'_1$ and $v'_3$. Then $v_1v_2v_3v_4v'_1wv'_3v_6v_1$ is a $C_8$, a contradiction.

Finally, if $v'_3\in \{v_1,v_3\}$, then the $6$-cycle $v_1uv_3v_4v'_1v_6v_1$ contains two fat pairs in one class and a fat pair in the other class, contradicting Claim \ref{fatfat1}.
\end{proof}

Let us fix a fat pair $v_1,v_3$ in one of the parts, $A$ and consider the union of the neighborhoods of the corresponding fat pairs in the other part, $B$. Let $g(v_1,v_3)$ denote its cardinality, i.e. 
$$g(v_1,v_3):=\sum_{\substack{v_4,v_6 \textrm{ is a fat pair in } B \\ v_1,v_3,v_4, v_6 \textrm{ are contained in a fat } C_6}}|N(v_4,v_6)|.$$

\begin{claim} \label{fv_1v_3}
For any fat pair $v_1,v_3$ we have
$$g(v_1,v_3)\le 4n.$$
\end{claim}

\begin{proof} By Claim \ref{v_1v_3} we know that the fat pairs $v_4, v_6$ from $B$ that are contained in a fat $6$-cycle $v_1v_2v_3v_4v_5v_6v_1$ must pairwise intersect, so the auxiliary graph $G_0$ containing these fat pairs as edges is either a star or a triangle.

\ 

$\bullet$ If $G_0$ is a triangle, we are done by using that $|N(u,v)|\le n$ for any $u, v$.

\ 

$\bullet$ If $G_0$ is a star with center $x$, then for every fat pair $v_4,v_6$ either $x = v_4$ or $x = v_6$. Let $G_1$ be the graph consisting of those edges where $x = v_4$ and $G_2$ be the graph consisting of those edges where $x = v_6$. 

\begin{obs}
Suppose that $v_1,v_3$ is a fat pair in class $A$, while $v_4,v_6$ and $v_4,v'_6$ are distinct fat pairs from class $B$ such that $v_1v_2v_3v_4v_5v_6v_1$ and $v_1v'_2v_3v_4v'_5v'_6$ are fat $6$-cycles. Then we have $$(N(v_4,v_6) \cap N(v'_4,v'_6)) \setminus \{v_1,v_3\} = \emptyset.$$ 

\end{obs}

\begin{proof} 
Let us suppose by contradiction that there is $x \in (N(v_4,v_6) \cap N(v_4,v'_6)) \setminus \{v_1,v_3\}$. By the fatness of the pair $v_1,v_3$ there is $z \in N(v_1,v_3) \setminus \{v_4,v_6,v'_6\}$ and by the fatness of the pair $v_4, v_6$, we can find $y \in N(v_6,v_4) \setminus \{v_1,v_3,x\}$. Then $v_1zv_3v_4yv_6xv'_6v_1$ is a $C_8$, a contradiction.
\end{proof}

This implies that for the fat pairs in $G_1$ we have $$\sum_{\substack{v_4,v_6 \textrm{ is a fat pair in } G_1 \\ v_1,v_3,v_4, v_6 \textrm{ are contained in a fat } C_6}}|N(v_4,v_6)\setminus \{v_1,v_3\}| \le n,$$
as we add up the cardinalities of disjoint sets. The same statement is true for $G_2$. This implies $$\sum_{\substack{v_4,v_6 \textrm{ is a fat pair in } G_0 \\ v_1,v_3,v_4, v_6 \textrm{ are contained in a fat } C_6}}|N(v_4,v_6)\setminus \{v_1,v_3\}| \le 2n.$$

Observe that there are less than $n$ edges in the star $G_0$, thus we subtract the two elements $v_1$ and $v_3$ less than $n$ times altogether. This finishes the proof.


\end{proof}

\noindent
Now consider the hypergraph $\cH$ whose vertex set is $B$ and its edge set is $$\{N(v_1,v_3) : \{v_1,v_3\} \subset A \textrm{ is a fat pair that is contained in at least one fat }C_6\}.$$

Recall that a Berge-$C_4$ in a hypergraph is an alternating sequence of distinct vertices and hyperedges of the form $v_1$,$h_{1}$,$v_2$,$h_{2}$,$v_3$,$h_3$,$v_4$,$h_4$,$v_1$ where $v_i,v_{i+1} \in h_{i}$ for each $i \in \{1,2,3\}$ and $v_4,v_1 \in h_4$.


\begin{claim} $\cH$ is Berge-$C_4$-free.

\end{claim}

\begin{proof} We prove it by contradiction. Let us suppose that we have a Berge-$C_4$ in $\cH$ that is a sequence $$x, N(a(x,y),b(x,y)), y , N(a(y,z),b(y,z)), z, N(a(z,w),b(z,w)),w, N(a(w,x),b(w,x)),$$ where we have:

\begin{enumerate}
\item[(1)] $x,y,z,w \in B$ are distinct vertices,

\item[(2)] $a(x,y),b(x,y),a(y,z),b(y,z),a(z,w),b(z,w),a(w,x),b(w,x) \in A,$ and

\item[(3)] $xa(x,y)y,xb(x,y)y,ya(y,z)z,yb(y,z)z,za(z,w)w,zb(z,w)w,wa(w,x)x,wb(w,x)x$ are all cherries.
\end{enumerate}

Let us consider the subgraph $G'$ of our original graph $G$ that is spanned by the vertices $x,y,z,w$, $a(x,y),b(x,y),a(y,z),b(y,z),a(z,w),b(z,w),a(w,x),b(w,x)$.
If one can find a matching in this bipartite graph that covers $x,y,z$ and $w$, that would immediately give a $C_8$, which is a contradiction. So to get the desired contradiction we apply Hall's theorem \cite{H1935} and check Hall's condition holds for the set of vertices $\{x,y,z,w\}$. 

Let us consider a subset $X\subseteq \{x,y,z,w\}$. Each vertex in $\{x,y,z,w\}$ is connected to at least two vertices in $G'$, so we are done if $|X|\le 2$. If $|X|=3$, we can assume without loss of generality that $X=\{x,y,z\}$. Then $x,y$ and $z$ are connected to the same 2 vertices. This implies $\{a(x,y),b(x,y)\}=\{a(y,z),b(y,z)\}$, thus two hyperedges of the Berge-$C_4$ coincide, a contradiction. Finally, if $|X|=4$, then there are at most three vertices in the other part of $G'$, thus there are at most 3 different hyperedges in the Berge-$C_4$, a contradiction.
\end{proof}

The above claim together with Corollary \ref{gyorilemons} implies that \begin{equation}\label{bergeuj}\sum_{\substack{v_1,v_3 \textrm{ is a fat pair in $A$ that is} \\ \textrm{contained in at least one fat }C_6}} |N(v_1,v_3)|=O(n^{1.5}). \end{equation}

\begin{claim} There are
$O(n^{2.5})$ fat $6$-cycles in $G$.

\end{claim}

\begin{proof}
Observe first that we obtain an upper bound on the number of fat $6$-cycles if we pick the fat pairs $v_1,v_3$ and $v_4,v_6$, then multiply the number of vertices that can be $v_2$ with the number of vertices that can be $v_5$. These later quantities we can upper bound by $|N(v_1,v_3)|$ and $|N(v_4,v_6)|$, as $v_2$ is connected to both $v_1$ and $v_3$, while $v_5$ is connected to both $v_4$ and $v_6$. Thus the number of fat $6$-cycles is at most

$$\sum_{\substack{v_1,v_3 \textrm{ and }v_4,v_6 \textrm{ are fat} \\ v_1,v_3,v_4, v_6 \textrm{ are contained in a fat } C_6}} |N(v_1,v_3)||N(v_4,v_6)| = $$

$$\sum_{v_1,v_3 \textrm{ is a fat pair in $A$ }} |N(v_1,v_3)| \ \  \cdot \sum_{\substack{v_4,v_6 \textrm{ is a fat pair} \\ v_1,v_3,v_4, v_6 \textrm{ are contained in a fat } C_6}} |N(v_4,v_6)| \ \le $$ 

$$4n \cdot \sum_{\substack{v_1,v_3 \textrm{ is a fat pair in $A$ that is} \\ \textrm{contained in at least one fat }C_6}} |N(v_1,v_3)|= 4n \cdot O(n^{1.5}), $$

where we use (\ref{bergeuj}) in the last inequality.

\end{proof}

Now we start to analyze those $6$-cycles that contain fat pairs just from one side, say $B$. We show that the number of these $6$-cycles is at most $6 \binom{|A|}{3}+O(|A|^{2.5})$. By symmetry it is enough to finish the proof of Theorem \ref{bipC6C8}, since $|A|+|B| \ \le n$ implies $6 \binom{|A|}{3}+ 6\binom{|B|}{3} \le 6\binom{n}{3}=n^3+O(n^2)$.


For $x,y,z\in A$ let $h(x,y,z)$ denote the number of $6$-cycles containing them in $G$. We call a set $\{x,y,z\}\subset A$ \textit{good} if $h(x,y,z)\le 6$, and \textit{bad} otherwise. Note that if there are no fat pairs among $x,y,z$, then $h(x,y,z)\le 27$.
We say that a $4$-set $F\subset A$ is \textit{nice} if the four vertices have three common neighbors in $B$, forming a copy of $K_{4,3}$ in $G$. It is easy to see that two vertices from a nice set having another common neighbor would create a $C_8$. This implies that $3$-subsets of a nice set are good.

We say a pair $x,y\in A$ is \textit{marked}, if they have exactly three common neighbors. 

\begin{claim}\label{da1} Let us assume there is a six-cycle $xv_1yv_2zv_3x$ that contains a marked pair $x,y$, such that $x$ and $y$ are contained in a nice set $S$, but $z$ is not contained in $S$. Then $\{x,y,z\}$ is a good triple.
\end{claim}

\begin{proof} Let us pick $s\in S\setminus\{x,y\}$. By the definition of the nice set, there is a $3$-set $S'\subset B$ consisting of the common neighbors of the vertices in $S$. If $v_3\not \in S'$, then $yv_3zv_2xtst'y$ is a $C_8$, where $t$ and $t'$ are arbitrary elements of $S'$, different from $v_2$. This is a contradiction. Similarly we can obtain a contradiction if $v_2\not\in S'$ by symmetry. If $v_3\in S'$ but $v_1\not \in S'$, then $v_3zv_2xv_1ytsv_3$ is a $C_8$, where $t\in S'\setminus \{v_2,v_3\}$.

Thus we have $S'=\{v_1,v_2,v_3\}$. Let us assume there is $w\in B\setminus S'$ that is connected to both $z$ and $x$. Then $zwxv_1sv_2yv_3z$ is a $C_8$, a contradiction. Similarly $z$ and $y$ cannot have another common neighbor, just like $x$ and $y$. Thus all the $6$-cycles containing $x$, $y$ and $z$ have to contain $v_1,v_2,v_3$, so there are at most $6$ of them.
\end{proof}

\begin{claim}\label{da2}  The number of marked pairs that are not in a nice set is $O(|A|^{1.5})$.

\end{claim}

\begin{proof} We prove this by showing that there is no $C_4$ in the auxiliary graph consisting of marked edges that are not in nice sets. Indeed, assume there is a $C_4$, $xyzwx$. Let us consider the auxiliary bipartite graph consisting of the marked pairs $xy,yz,zw,wx$ on one side and the common neighbors of these pairs on the other side, where a marked edge is connected to the three common neighbors defining it. A matching covering $xy,yz,zw,wx$ would correspond to a $C_8$ $xv_yv_2zv_3wv_4$ in $G$, where the $v_i$s are the other vertices in the matching. If there is no such matching, then by Hall's condition the four vertices $x,y,z,w$ have the same three common neighbors, thus they form a nice set, a contradiction.

\end{proof}


This implies that there are $O(|A|^{2.5})$ copies of $C_6$ containing any marked pairs not in a nice set. 
Hence from now on we consider only non-marked pairs. Let us consider a triple in $A$ not containing any fat or marked pair. 

\begin{claim}\label{da3} There are $O(|A|^2)$ bad triples containing neither a fat, nor a marked pair.

\end{claim}

\begin{proof} Let $\{x,y,z\}$ be such a triple. In every $6$-cycle containing $x,y,z$, there is a vertex connected to both $x$ and $y$, another connected to both $y$ and $z$ and a third one connected to both $z$ and $x$. As none of these pairs is marked, there are at most two candidates for each position. If any of them would coincide, the number of $6$-cycles containing $x,y,z$ would be at most $4$. Thus there are $6$ vertices in $B$ $a_1,a_2,b_1,b_2,c_1,c_2$ such that $a_1$ and $a_2$ are connected to $x$ and $y$, $b_1$ and $b_2$ are connected to $y$ and $z$, while $c_1$ and $c_2$ are connected to $z$ and $x$.

Let us consider the auxiliary $3$-uniform hypergraph $\cH_0$ of bad triples containing no marked pair. We claim it is a linear hypergraph, finishing the proof. Indeed, otherwise without loss of generality there is another bad triple $\{x',y,z\}$ with the corresponding vertices $a_1',a'_2,b_1,b_2,c'_1,c'_2$ in $B$. We can assume that $a_1\neq a_1'$ and $c_1\neq c'_1$, and we have $a_1\neq c_1'$ (similarly $a_1'\neq c_1$), as otherwise $x$, $y$ and $z$ (or $x'$, $y$ and $z$) would have a common neighbor. Then $xa_1ya'_1x'c_1'zc_1x$ is a $C_8$ in $G$, a contradiction.
\end{proof}



The above claims together imply that there are $O(|A|^{2.5})$ bad triples. Indeed, there are $O(|A|^2)$ containing no marked pair by Claim \ref{da3}. There are two kinds of marked pairs. There are $O(|A|^{1.5})$ of those that are not in a nice set by Claim \ref{da2}, thus they are contained in at most $O(|A|^{2.5})$ triples. Finally, if a marked pair is contained in a nice set, then it is not contained in a bad triple.

The total number of $6$-cycles that do not contain a fat pair from part $A$ is $$\sum_{x,y,z\in A} h(x,y,z)=\sum_{x,y,z\in A \textnormal{ is a good triple}} h(x,y,z)+\sum_{x,y,z\in A \textnormal{ is a bad triple}} h(x,y,z) \le$$

$$\sum_{x,y,z\in A \textnormal{ is a good triple}} 6+\sum_{x,y,z\in A \textnormal{ is a bad triple}} 27 \le 6\binom{|A|}{3}+27 \cdot O(|A|^{2.5}).$$

This finishes the proof of the theorem.

	\section{Forbidding a set of cycles}
	
	In this section we study the case when multiple cycles are forbidden. Recall that if $A$ is a set of integers, such that each integer is at least 3, then the set of cycles $\cC_A=\{C_a: a\in A\}$, $A_e$ is the set of even numbers in $A$ and $A_o$ is the set of odd numbers in $A$. 
	
	\subsection{Basic Lemmas}
    The following simple lemma shows that if we count copies of an even cycle of given length, then forbidding odd cycles does not change the order of magnitude.

\begin{lemma}\label{paros} If $2k\not\in A$, then $$ex(n,C_{2k},\cC_A)=\Theta(ex(n,C_{2k},\cC_{A_e})).$$

\end{lemma}

\begin{proof} It is obvious that $ex(n,C_{2k},\cC_A)\le ex(n,C_{2k},\cC_{A_e})$, as a $\cC_A$-free graph is also $\cC_{A_e}$-free. Let $G$ be a $\cC_{A_e}$-free graph. We are going to show that it has a $\cC_A$-free subgraph $G'$ that contains at least $1/2^{2k-1}$ fraction of the $2k$-cycles of $G$, finishing the proof.

Let us consider a random $2$-coloring of the vertices of $G$, where every vertex becomes blue with probability $1/2$, and red otherwise. Let us delete the edges inside the color classes, and let $G'$ be the resulting graph. As $G'$ is bipartite, it does not contain any cycle in $\cC_{A_o}$. The probability that a $2k$-cycle of $G$ is also in $G'$ is $1/2^{2k-1}$, as the first vertex can be of any color, but then the color of all the other vertices is determined. Thus the expected number of $2k$-cycles in $G'$ is $1/2^{2k-1}$ fraction of the $2k$-cycles in $G$, hence there is a $2$-coloring with at least that many $2k$-cycles.
\end{proof}

Next we show that if we count copies of an odd cycle of given length, then forbidding shorter odd cycles does not change the order of magnitude.

\begin{lemma}\label{paratlan} Let $O_k$ be the set of odd integers less than $2k+1$. Then $$ex(n,C_{2k+1},\cC_A)=\Theta(ex(n,C_{2k+1},\cC_{A\setminus O_k})).$$

\end{lemma}

\begin{proof} The proof goes similarly to that of the previous lemma, one of the directions is again trivial. Let $G$ be a $\cC_{A \setminus O_k}$-free graph. We are going to show that it has a $\cC_A$-free subgraph $G'$ that contains at least a constant fraction of the $2k$-cycles of $G$.

Let us consider a random partition of the vertices of $G$ into $2k+1$ classes $V_1,\dots,V_{2k+1}$, where each vertex goes to each class with the same probability $1/(2k+1)$. We keep the edges between $V_i$ and $V_{i+1}$ for $i\le 2k$, and the edges between $V_{2k+1}$ and $V_1$. We delete all the other edges and let $G'$ be the resulting graph. It is easy to see that if we delete $V_i$ from $G'$, we obtain a bipartite graph, hence an odd cycle has to contain a vertex from $V_i$, for every $i\le 2k+1$. This means every odd cycle has length at least $2k+1$.

It is left to prove that $G'$ contains many $(2k+1)$-cycles. An arbitrary cycle in $G$ is a cycle in $G'$ with probability $1/(2k+1)^{2k}$, finishing the proof.
\end{proof}

\begin{lemma}\label{utso} Let $m$ and $s$ be fixed positive integers, and let $G$ be a graph on $n$ vertices. Suppose there is a partition of $V(G)$ into sets $V_1,\dots, V_s$ satisfying the following properties:

\textbf{(i)} there is no $P_3$ with both endpoints in $V_i$ for $i<s$,

\textbf{(ii)} there is no $P_{m+1}$ with endpoints in $V_i$ and $V_j$ if $i\neq j$, and

\textbf{(iii)} $V_s$ is an independent set.

\smallskip 

Then $|E(G)|=O(n)$.

\end{lemma}

\begin{proof} Let us first delete every vertex with degree at most $m$. Then repeat this procedure until we obtain a graph with minimum degree greater than $m$, or a graph with no vertices. We will show that the resulting graph has no vertices, which would imply that $G$ has at most $mn = O(n)$ edges, since obviously at most $mn$ edges were deleted. 

If the resulting graph contains a vertex, it has to contain a vertex $x_1 \in V_j$ with $j\neq s$ (because $V_s$ is an independent set). Starting from $x_1$, we build a path $P_m$ greedily. First we pick a neighbor $x_2$ of $x_1$. For each $2 \le i \le m-1$, after picking $x_i$, we pick a neighbor $x_{i+1}$ of it that has not appeared in the path; it forbids at most $i-1<m$ neighbors, thus we can pick such a neighbor. After picking $x_m$, we add one more condition: the neighbor of $x_m$ we pick as $x_{m+1}$ should not be in $V_j$. As $x_m$ has at most one neighbor in $V_j$ by \textbf{(i)}, this is at most one more forbidden neighbor, so we can still find one satisfying all these conditions. This way we obtain a $P_{m+1}$ with endpoints in different parts, a contradiction with \textbf{(ii)}. \end{proof}


\subsection{Proof of Theorem \ref{longer_kor}: Forbidding an additional cycle in a graph of given even girth}
Now we prove Theorem \ref{longer_kor}. We restate it here for convenience.

	\begin{thm*}\label{longer_kor_restated} For any $k > l$ and $m\ge 2$ such that $2k \neq ml$ we have $$ex(n,C_{ml},\cC_{2l-1} \cup \{C_{2k}\})=\Theta(n^m)$$
    \end{thm*}

\begin{proof} By Lemma \ref{paratlan} and \ref{paros} it is enough to prove $$ex(n,C_{ml},\cC_{2l-2} \cup \{C_{2k}\})=\Theta(n^m).$$



The lower bound is given by the theta-$(n,C_m,l)$ graph.
It contains $\Omega(n^m)$ cycles of length $ml$, and additionally contains only cycles of length $2l$.

\vspace{3mm}

First we prove the upper bound for $m=2$.
We consider a graph $G$ on $n$ vertices that does not contain any of the forbidden cycles. We can assume it is bipartite by Lemma \ref{paros}.
Then $G$ has girth at least $2l$, thus it is easy to see that for any vertices $u,v$ and any length $i\le l-1$, there is at most one path of length $i$ between $u$ and $v$.
    \begin{claim}\label{metszet} If $C$ and $C'$ are $2l$-cycles in $G$ sharing a path of length $l-1$, then their intersection is exactly a path of length $l-1$ or $l$.
     
    \end{claim}
    
    \begin{proof} Observe that if we can find a closed walk of length less than $2l$ which is not contained in a tree, then it implies the existence of a cycle of length less than $2l$, a contradiction. 
    
    Let us consider the longest path $Q=u_1\dots u_i$ shared by $C$ and $C'$. If $Q$ has length more than $l$, then there are paths of length less than $l$ in both $C$ and $C'$ with endpoints $u_1$ and $u_i$, and these two paths cannot be the same. Thus they form a closed walk of length less than $2l$.
    
    Let $v\in (V(C)\cap V(C'))\setminus V(Q)$ be the vertex that is the closest to $u_1$ in $C$. Let $x$ (resp. $x'$) be the distance between $v$ and $u_1$ in $C$ (resp. $C'$). Suppose first that $x \not = x'$.  Without loss of generality, we may suppose $x<x'$. Then the subpath of length $x$ between $u_1$ and $v$ in $C$, the subpath of length $2l-(i-1)-x'$ between $v$ and $u_i$ in $C'$, and the path $Q$ of length $i-1$ between $u_i$ and $u_1$ form a closed walk of length less than $2l$.
    

Hence we can assume $x=x'$. If $x<l$, then the paths of length $x$ between $u_1$ and $v$ in $C$ and $C'$ form a closed walk of length less than $2l$. If $x\ge l$, then either $v \in V(Q)$ or is adjacent to $u_l$, contradicting either our assumption that $v\in (V(C)\cap V(C'))\setminus V(Q)$ or our assumption that $Q$ was the longest shared path.
	\end{proof}	
    
The proof of Theorem \ref{longer_kor} goes similarly to the proof of Theorem \ref{four_cycle} from here. We call a pair of vertices \textit{fat} if there are at least $4l^2$ paths, each of length $l$ (i.e., $l$ edges) between them. We call a copy of $C_{2l}$ \textit{fat} if all the $l$ pairs of opposite vertices are fat. 
		
		\begin{claim}\label{dagi} Let $\{u,v\}$ be a fat pair and $X$ be a set of at most $4l$ vertices. Then there is a path $P$ of length $l$ between $u$ and $v$ such that $(V(P) \setminus \{u,v\}) \cap X = \emptyset$.
			
		\end{claim}
		
		\begin{proof} For every $i\le l-1$, there are at most $4l$ paths $uu_1 \ldots u_i \ldots u_{l-1}v$ such that $u_i$ is in $X$. Indeed, there are at most $4l$ ways to choose $u_i$ from $X$, and after that there is only one choice for the remaining vertices, because there is at most one path of length $i$ from $u$ to $u_i$, and at most one path of length $l-i$ from $u_i$ to $v$. Since there are $l-1$ ways to choose $i$, altogether there are at most $4l(l-1)$ paths intersecting $X$, finishing the proof.
		\end{proof}
		
		Observe now that the number of non-fat $C_{2l}$'s is at most $\binom{4l^2-1}{2}\binom{n}{2}$, as there are at most $\binom{n}{2}$ non-fat pairs and each of them is contained in at most $\binom{4l^2-1}{2}$ $C_{2l}$'s as an opposite pair. This way we count every non-fat $C_{2l}$ at least once.
		
		Let us only consider fat $C_{2l}$'s from now on. We go through them in an arbitrary order, one by one, and pick exactly one path $u_1u_2 \dots u_l$ of length $l-1$ from each of them (from the $2l$ paths of length $l-1$ in the $C_{2l}$); we always pick the path which was picked the smallest number of times before (in case there is more than one such path, then we pick one of them arbitrarily). 
		
		After this procedure, every path $Q$ of length $l-1$ is picked a certain number of times. Let us denote this number by $m(Q)$, and we call it the \textit{multiplicity} of $Q$. Note that adding up $m(Q)$ for all the paths $Q$ of length $l-1$ gives the number of fat $C_{2l}$'s in $G$. 
Assume first that at the end of this algorithm $m(Q)\le \frac{2k(k-l)}{l-2}$ for every path $Q$ of length $l-1$. Then the number of fat $C_{2l}$'s is at most $ \frac{2k(k-l)}{l-2}$ times the number of the paths of length $l-1$, which is at most $ \frac{2k(k-l)}{l-2}\binom{n}{2}$, as there is at most one path of length $l-1$ between any two vertices.
		
		Hence we can assume there is a path $Q$ of length $l-1$ with multiplicity greater than $ \frac{2k(k-l)}{l-2}$. We claim that in this case there is a copy of $C_{2k}$ in $G$, which leads to a contradiction, finishing the proof. More precisely, we are going to prove the following statement: 
		
		\begin{claim}
		For every $l \le r \le k$ there is a $C_{2r}$ in $G$, that contains a path $Q_r$ of length $l-1$ with $$m(Q_r)\ge  \frac{2k(k-r)}{l-2}.$$ 
		\end{claim}
		\begin{proof}
		We prove it by induction on $r$. More precisely, are going to assume that the statement is true for $r$ and show that it is true for $r+l-1$. Therefore, we need to start with the base cases $l \le r\le 2l-2$. For the base case $r = l$, consider any $C_{2l}$ containing $Q = Q_l$. 
         For the other base cases consider $Q=u_1u_2 \dots u_l$ with $m(Q)\ge  \frac{2k(k-l)}{(l-2)}$. We have two fat $C_{2l}$'s, say $C$ and $C'$, containing $Q$ such that every subpath of length $l-1$ in each of them has multiplicity at least $ \frac{2k(k-l)}{(l-2)}-2$. By Claim \ref{metszet} the intersection of $C$ and $C'$ is a path $Q'$ of length $l$ or $l-1$. It means either $Q=Q'$ or  $Q'$ consists of $Q$ plus an additional vertex adjacent to either $u_1$ or $u_l$.
        
        Note that for any pair $u, v$ of vertices that are opposite in either $C$ or $C'$, there is a path $P(u,v)$ of length $l$ between $u$ and $v$ such that $V(P) \setminus \{u, v\} \cap (V(C) \cup V(C')) = \emptyset$, by Claim \ref{dagi}, since $u, v$ is a fat pair.
        
        Let us assume first that $Q'=Q$. Let $ 1 \le i \le l-2$ and let $w$ be the vertex opposite to $u_i$ in $C'$. Consider the subpath of $Q$ from $u_1$ to $u_i$, the path $P(u_i, w)$ of length $l$ from $u_i$ to $w$, the subpath of length $i$ from $w$ to $u_l$ in $C'$, and the path of length $l+1$ from $u_l$ to $u_1$ in $C$. They form a cycle of length $2l+2i$, that contains a subpath of $C$ of multiplicity at least $$ \frac{2k(k-l)}{l-2}-1\ge  \frac{2k(k-l-i)}{l-2}.$$

        
        If $Q'\neq Q$, we can assume without loss of generality that $Q'=u_1u_2 \dots u_lu_{l+1}$.  Let $ 1 \le i \le l-3$ and let $w'$ be the vertex opposite to $u_{i+1}$ in $C'$. Consider the subpath of $Q'$ from $u_1$ to $u_{i+1}$, the path $P(u_{i+1}, w')$ of length $l$ from $u_{i+1}$ to $w'$, the subpath of length $i$ from $w'$ to $u_{l+1}$ in $C'$, and the path of length $l$ (different from $Q'$) from $u_{l+1}$ to $u_1$ in $C$. They form a cycle of length $2l+2i$, that contains a subpath of $C$ of multiplicity at least $$ \frac{2k(k-l)}{l-2}-1\ge  \frac{2k(k-l-i)}{l-2},$$ finishing the base cases.

		Let us continue with the induction step. Assume we are given a cycle $C$ of length $2r$ that contains a path $Q_r=u_1u_2 \dots u_l$ with $m(Q)\ge  \frac{2k(k-r)}{l-2}$, and we are going to find a cycle of length $2r+2l-2$ that contains a path of length $l-1$ with multiplicity at least $  \frac{2k(k-r-l+2)}{l-2}$. For any $i \le  \frac{2k(k-r)}{l-2}$, when $Q_r$ was picked for the $i$th time, the corresponding fat $C_{2l}$ contained $2l$ paths of length $l-1$ each of which had been picked earlier at least $i-1$ times, thus they have multiplicity at least $i-1$. Let $\cF_r$ be the set of those fat $C_{2l}$'s where $Q_r$ was picked for the last $$\frac{2k(k-r)}{l-2}- \frac{2k(k-r-l+2)}{l-2}=2k$$ times, so $|\cF_r|\ge 2k$. All the other paths of length $l-1$ in each of these fat $C_{2l}$'s have multiplicity at least $$\frac{2k(k-r)}{l-2}-2k=\frac{2k(k-r-l+2)}{l-2}.$$ 
        
        First observe that for every vertex $w\in V(C)\setminus V(Q_r)$ there is at most one cycle in $\cF_r$ which contains $w$ such that $w$ is neither next to $u_1$, nor to $u_l$ in that cycle. Indeed, two such cycles would contradict Claim \ref{metszet}. As there are less than $2k$ choices for $w$, either there exists a cycle $C'=u_1u_2 \dots u_lv_1\dots v_lu_1 \in \cF_r$ such that all of the vertices $v_1,\dots, v_l$ are not in $C$ or there exists a cycle $C'=u_1u_2 \dots u_lv_1\dots v_lu_1 \in \cF_r$ such that only $v_l$ or $v_1$ is in $C$ among $v_1,\dots, v_l$. Without loss of generality, suppose $v_1$ is in $C$.
	
    Consider a path $P$ of length $l$ from $u_2$ to $v_2$ that avoids $(V(C) \cup V(C')) \setminus \{u_2, v_2\}$ (such a path exists because $u_2, v_2$ is an opposite pair in $C'$, thus it is a fat pair and we can apply Claim \ref{dagi}). Let $P'$ be a subpath of length $l-1$ in $C'$ from $v_2$ to $u_1$. We replace the edge $u_1u_2$ in $C$ by the union of paths $P$ and $P'$. Note that we replaced an edge with a path of length $2l-1$, so the resulting cycle has length $2r+ 2l-2$ and it contains the subpath $v_2\dots v_lu_1$, which is a subpath of $C'$, thus it has multiplicity at least $$\frac{2k(k-r-l+2)}{l-2},$$ as required.
	\end{proof}

We are done with the case $m=2$, now we consider the case $m$ is larger.
Let $G$ again be a graph that does not contain $C_3,C_4, \dots, C_{2l-2}, C_{2k}$.
From here we follow the proof of Theorem \ref{main2}. First we introduce the following notation. For two distinct vertices $a, b \in V(G)$, let $$f_l(a,b):= \textrm{ number of paths of $l$ edges between } a \textrm{ and } b.$$ 
In particular $f_2(a,b)=f(a,b)$. Then we have \begin{equation}\label{equ1}\frac{1}{2}\sum_{a\neq b,\,a,b \in V(G)}\binom{f_l(a,b)}{2}=O(n^2),\end{equation} by the case $m=2$, since the left-hand side is equal to the number of $C_{2l}$'s in $G$.

		\begin{claim}\label{klem} For every $a\in V(G)$ we have $$\sum_{b \in V(G) \setminus \{a\}} f_l(a,b)=O(n).$$
			
		\end{claim}
        
\begin{proof}
Notice that for any $i < l$, the set $N_i(a)$ does not contain any edges. It is easy to see that $\sum_{b \in V(G) \setminus \{a\}} f_l(a,b)$ is equal to the number of edges between $N_{l-1}(a)$ and $N_l(a)$. We will show this number is $O(n)$ by using Lemma \ref{utso} to the bipartite graph $G'$ with vertex set $N_{l-1}(a)\cup N_l(a)$ and the edge set being the set of edges of $G$ between $N_{l-1}(a)$ and $N_l(a)$.

Let $w_1,\dots, w_{s-1}$ be the neighbors of $v$, and let $V_i=N_{l-1}(a)\cap N_{l-2}(w_i)$ for $1 \le i \le s-1$. Let $V_s=N_l(a)$. It is easy to see that $V_1, V_2, \ldots, V_s$ partition $V(G')$. Observe that a $P_3$ in $G'$ with both endpoints inside $V_i$ (for $i < s$) would create a cycle of length at most $2l-2$.
This implies \textbf{(i)} of Lemma \ref{utso} is satisfied. A path $P_{2k-2l+3}$ in $G'$ with endpoints in $V_i$ and $V_j$ with $i\neq j$ would create a cycle of length $2k$ in $G$, which shows that \textbf{(ii)} of Lemma \ref{utso} is satisfied. Since $G'$ is bipartite, clearly $V_s$ is independent in $G'$, thus we can apply Lemma \ref{utso}, finishing the proof of the claim.
\end{proof}


		
		Let us fix vertices $v_1, v_2, \dots, v_m$ and let $g(v_1,v_2, \ldots, v_m)$ be the number of $C_{ml}$'s in $G$ where $v_i$ is $li'$th vertex ($i\le m$). 
		 Clearly $g(v_1,v_2, \ldots, v_m) \le \prod_{j=1}^m f_l(v_j,v_{j+1})$ (where $v_{l+1}=v_1$ in the product).

		 If we add up $g(v_1,v_2, \ldots, v_m)$ for all possible $m$-tuples $(v_1, v_2, \dots, v_m)$ of $l$ distinct vertices in $V(G)$, we count every $C_{ml}$ exactly $4m$ times. It means the number of $C_{ml}$'s is at most
		
		\begin{equation}\label{equ3}\frac{1}{4m}\sum_{(v_1, v_2, \dots, v_m)}\prod_{j=1}^m f_l(v_jv_{j+1})\le\frac{1}{4m}\sum_{(v_1, v_2, \dots, v_m)}\frac{f_l^2(v_1,v_2)+f_l^2(v_2,v_3)}{2}\prod_{j=3}^m f_l(v_jv_{j+1}).
		\end{equation}
		
		Fix two vertices $u,v\in V(G)$ and let us examine what factor $f_l^2(u,v)$ is multiplied with in \eqref{equ3}. It is easy to see that $f_l^2(u,v)$ appears in \eqref{equ3} whenever $u=v_1,v=v_2$ or $u=v_2,v=v_1$ or $u=v_2,v=v_3$ or $u=v_3,v=v_2$.
		Let us start with the case $u=v_1$ and $v=v_2$. 
		Then $f_l^2(u,v)$ is multiplied with $$\frac{1}{8m}\left(\prod_{j=3}^{m-1} f_l(v_jv_{j+1})\right)f(v_m,u) = \frac{1}{8m}f_l(u,v_m) \prod_{j=3}^{m-1} f_l(v_jv_{j+1})$$ for all the choices of tuples $(v_3, v_4, \dots, v_m)$ (where these are all different vertices). We claim that $$\sum_{(v_3, v_4, \dots, v_m)}\frac{1}{8m}f_l(u,v_m)\prod_{j=3}^{m-1} f_l(v_jv_{j+1})\le \frac{(2k-2)^{m-2}n^{m-2}}{8m}.$$ Indeed, we can rewrite the left-hand side as $$\frac{1}{8m}\sum_{v_l\in V(G)}f_l(u,v_m)\sum_{v_{m-1}\in V(G)}f_l(v_m,v_{m-1})\dots \sum_{v_{3}\in V(G)}f_l(v_4,v_{3}),$$ and each factor is at most $O(n)$ by Claim \ref{klem}.

		Similar calculation in the other three cases gives the same upper bound, so adding up all four cases, we get an additional factor of 4, showing that the number of copies of $C_{ml}$ is at most
		$$\frac{1}{ml}\sum_{u \not = v, u,v\in V}f_l^2(u,v)O(n^{m-2}).$$

		Finally observe that $$\sum_{u \not = v, u,v\in V}f_l^2(u,v)=O(n^2),$$
by (\ref{equ1}) and Claim \ref{klem}. This finishes the proof.
\end{proof}
	
\subsection{Proofs of Theorem \ref{morecycasy1} and Theorem \ref{morecycasy2}: Counting $C_4$'s or $C_6$'s when a set of cycles is forbidden}
    Below we determine the asymptotics of $ex(n,C_4,\cC_A)$ and the order of magnitude of $ex(n,C_6,\cC_A)$. For the convenience of the reader, we restate Theorem \ref{morecycasy1} and Theorem \ref{morecycasy2}.
	
	\begin{thm*}\label{negycycle} For any $k \ge 3$ we have
    
    $$ex(n,C_4,\cC_A)=\left\{ \begin{array}{l l}
		0 & \textrm{if $4\in A$}\\
		 (1+o(1)) \frac{(k-1)(k-2)}{4} n^2& \textrm{if $4\not\in A$ and $2k$ is the smallest element of $A_e$}\\
		(1+o(1)) \frac{1}{64} n^4 & \textrm{if $A_e=\emptyset$.}
		\end{array}
		\right.$$
	\end{thm*}
	
	\begin{proof} The first line is obvious. For the second line, the upper bound follows from Theorem \ref{four_cycle} as $C_{2k}$ is forbidden, while the lower bound is given by the complete bipartite graph $K_{k-1,n-k+1}$. For the third line, the lower bound is given by the graph $K_{n/2,n/2}$, while the upper bound follows from Theorem \ref{evenodd}.
		
	\end{proof}

	\begin{thm*} $$ex(n,C_6,\cC_A)=\left\{ \begin{array}{l l}
		0 & \textrm{if $6\in A$,}\\
		\Theta(n^2) & \textrm{if $6\not\in A$, $4\in A$ and $|A_e|\ge 2$,}\\
		\Theta(n^3) & \textrm{if $4,6\not\in A$ and $A_e\neq \emptyset$, or if $A_e=\{4\}$ }\\
		\Theta(n^6) & \textrm{if $A_e=\emptyset$.}
		\end{array}
		\right.$$
		
	\end{thm*}
	
	\begin{proof} The first line is obvious. For the second line, Theorem \ref{longer_kor} with $m=2$ and $l=3$ gives the upper bound and the lower bound. For the third line, the upper bound follows from Theorem \ref{main2} and if $4, 6 \not \in A$, then the lower bound is given by the graph $K_{3,n-3}$. 
    
    If $A_e=\{4\}$, the upper bound is given by Theorem \ref{main2}. For the lower bound let us consider the $C_4$-free graph $G$ given by \cite{GP2017,gs2017},  which contains $\Theta(n^3)$ $C_6$'s. $G$ might contain some forbidden odd cycles. However, Lemma \ref{paros} shows they do not change the order of magnitude.
    
		
		For the fourth line, the lower bound is given by the graph $K_{n/2,n/2}$, while the upper bound is obvious.
	\end{proof}

\section{Counting cycles in graphs with given odd girth}

\subsection{Proof of Theorem \ref{odd_girth}: Maximum number of $C_{2k+1}$'s in a graph of girth $2k+1$}
In this subsection we prove Theorem \ref{odd_girth}. 

Let $G$ be an almost-regular, $\{C_4,C_6, \ldots,C_{2k}\}$-free graph on $n$ vertices with $n^{1+1/k}/2$ edges given by Conjecture \ref{strongerEGC}. It follows that the degree of each vertex of $G$ is $n^{1/k}+O(1)$. We will show that $G$ contains at least $$(1-o(1)) \frac{1}{2k+1}\frac{n^{2+1/k}}{2}$$ copies of $C_{2k+1}$.

Consider an arbitrary vertex $v \in V(G)$. Recall that $N_i(v)$ denotes the set of vertices at distance $i$ from $v$. (Note that $N_1(v)$ is simply the neighborhood of $v$.) First we show the following. 


\begin{claim}
\label{branchout}
Let $2 \le i \le k$. Each vertex $u \in N_{i-1}(v)$ has at least $n^{1/k}+O(1)$ neighbors in $N_i(v)$. Moreover, no two vertices of $N_{i-1}(v)$ have a common neighbor in $N_i(v)$.
\end{claim}
\begin{proof}
Consider an arbitrary vertex $u \in N_{i-1}(v)$. If there are two edges $ux, uy$ with $x, y \in N_{i-2}(v)$, let $w$ be the first common ancestor of $x$ and $y$. Then the length of the cycle formed by the two paths from $w$ to $x$ and from $w$ to $y$ and the two edges $ux, uy$ is at most $2k$ and is even, a contradiction. So there is at most one edge from $u$ to the set $N_{i-2}(v)$. Now if there are two edges $ux, uy$ with $x, y \in N_{i-1}(v)$, then again consider the first common ancestor $w$, of $x$ and $y$ and we can find an even cycle of length at most $2k$ similarly. So the degree of $u$ in $G[N_{i-2}(v)]$ is at most one. Therefore, each vertex $u \in N_{i-1}(v)$ has at least $n^{1/k}+O(1)$ neighbors in $N_i(v)$ (recall the degree of $u$ is $n^{1/k}+O(1)$), proving the first part of the claim.

Suppose for a contradiction that there are two vertices $u, u' \in N_{i-1}(v)$, which have a common neighbor in $N_i(v)$. Then consider the first common ancestor of $u$ and $u'$, and we can again find an even cycle of length at most $2k$. This completes the proof of the claim. 
\end{proof}
The above claim implies the following.

\begin{claim}
\label{ntotheibyk}
 For all $1 \le i \le k$, we have $\abs{N_i(v)} = (1+o(1)) n^{i/k} $.
\end{claim}
\begin{proof}
Notice that Claim \ref{branchout} implies that there are at least $\abs{N_{i-1}(v)}(n^{1/k}+O(1))$ vertices in $N_i(v)$ for each $2 \le i \le k$. Since $N_1(v) = n^{1/k}+O(1)$, this proves the claim.
\end{proof}

Now we claim that there are $(1-o(1)) \frac{n^{1+1/k}}{2}$ edges of $G$ in $N_k(v)$ (i.e., basically all the edges of $G$ are in $N_k(v)$). Indeed, notice that $$\abs{N_k(v)} = (1+o(1)) n^{k/k} = (1+o(1)) n$$ by Claim \ref{ntotheibyk}, so $\abs{V(G) \setminus N_k(v)} = o(n)$. Therefore the sum of degrees of the vertices in $V(G) \setminus N_k(v)$ is $$o(n) \cdot (n^{1/k}+O(1)) = o(n^{1+1/k}),$$ showing that the number of edges incident to vertices outside $N_k(v)$ are negligible. This shows that there are $(1-o(1)) \frac{n^{1+1/k}}{2}$ edges in $G[N_k(v)]$, proving the claim.


Now we color each edge $ab$ of $G[N_k(v)]$ in the following manner: If the first common ancestor of $a$ and $b$ is not $v$, then $ab$ is colored with the color red, but if the only common ancestor of $a$ and $b$ is $v$ then it is colored blue. We want to show that most of the edges in $G[N_k(v)]$ are of color blue. To this end, let us upper bound the number of edges in $G[N_k(v)]$ of color red. 

Consider an arbitrary vertex $w \in N_1(v)$. Applying Claim \ref{branchout} repeatedly, one can obtain that $w$ has $$(n^{1/k}+O(1))^{k-1} = (1+o(1)) n^{(k-1)/k} $$ descendants in $N_k(v)$. By Theorem \ref{bs}, in the subgraph of $G$ induced by this set of descendants, there are at most $$O(n^{(k-1)/k})^{1+1/k} = (1+o(1)) O(n^{(k^2-1)/k^2}) $$ edges.

On the other hand, the end vertices of each red edge must have an ancestor $w \in N_1(v)$, so the total number of red edges is at most $$(1+o(1)) \abs{N_1(v)} O(n^{(k^2-1)/k^2})  = (1+o(1)) n^{1/k} O(n^{(k^2-1)/k^2})  = o(n^{1+1/k}).$$

This shows that there are $\frac{n^{1+1/k}}{2}(1-o(1))$ blue edges in $G[N_k(v)]$. Notice that any blue edge $ab$, together with the two paths joining $a$ and $b$ to $v$, forms a $C_{2k+1}$ in $G$ containing $v$. This shows that there are $\frac{n^{1+1/k}}{2}(1-o(1))$ copies of $C_{2k+1}$ in $G$ containing $v$. As $v$ was arbitrary, summing up for all the vertices of $G$, we get that there are at least $$(1-o(1))\frac{1}{2k+1}\frac{n^{2+1/k}}{2}$$ copies of $C_{2k+1}$ in $G$. 

Now it only remains to upper bound the number of $C_{2k+1}$'s in a graph $H$ of girth $2k+1$. For a pair $(v, xy)$ where $v \in V(H)$, and $xy \in E(H)$, there is at most one $C_{2k+1}$ in $H$ such that $xy$ is the edge in the $C_{2k+1}$ opposite to $v$. Indeed, there is at most one path of length $k$ in $H$ joining $v$ and $x$, and at most one path of length $k$ in $H$ joining $v$ and $y$, as $H$ has no cycles of length at most $2k$. On the other hand, a fixed $C_{2k+1}$ consists of $2k+1$ pairs $(v, xy)$ such that $v$ is opposite to an edge $xy$ of the cycle. Therefore, the number of $C_{2k+1}$'s in $H$ is at most $$\frac{n}{2k+1} \abs{E(H)} \le (1+o(1))\frac{n}{2k+1}\frac{n^{1+1/k}}{2}, $$ by Theorem \ref{AlonHL}.  This completes the proof. 





\subsection{Proofs of Theorem \ref{oddgirthpluseven}  and Theorem \ref{oddgirthplusodd}: Forbidding an additional cycle in a graph of odd girth}

In this subsection, we study the maximum number of $C_{2l+1}$'s in a $\cC_{2l} \cup \{C_{2k}\}$-free graph and the maximum number of $C_{2l+1}$'s in a $\cC_{2l} \cup \{C_{2k+1}\}$-free graph, and prove Theorem \ref{oddgirthpluseven} and Theorem \ref{oddgirthplusodd}.


If $l=1$ we count triangles in a $C_{2k}$-free graph or a $C_{2k+1}$-free graph. The second question was studied by Gy\H{o}ri and Li \cite{GyoriLi} and  Alon and Shikhelman \cite{ALS2016}. The first question was studied by Gishboliner and Shapira \cite{gs2017}. They showed the following. 

\begin{thm}[Gy\H{o}ri-Li, Alon-Shikhelman  and Gishboliner-Shapira]
For any $k \ge 2$ we have

\medskip 

\textbf{(i)} $\Omega(ex(n, \cC_{2k})) \le ex(n, C_3, C_{2k}) \le O_k(ex(n, C_{2k})).$

\smallskip

\textbf{(ii)} $\Omega(ex(n, \cC_{2k})) \le ex(n, C_3, C_{2k+1}) \le O(k \cdot ex(n, C_{2k})).$
\end{thm}

The above lower and upper bounds are known to be of the same order of magnitude, $\Theta(n^{1+1/k})$ when $k \in \{2, 3, 5\}$ (see \cite{B1966,W1991}).

In the rest of the section we study the case $l \ge 2$. 
For the lower bounds, we will use the following result of Ne\u set\u ril and R\"odl \cite{N_Rodl}. Girth of a hypergraph $H$ is defined as the length of a shortest Berge cycle in it. More formally, it is the smallest integer $k$ such that $H$ contains a Berge-$C_k$.
\begin{thm}[Ne\u set\u ril, R\"odl]
\label{Nes_Rodl}
For any positive integers $r \ge 2$ and $s \ge 3$, there exists an integer $n_0$ such that for all $n \ge n_0$, there is an $r$-uniform hypergraph on $n$ vertices with girth at least $s$ and having at least $n^{1+1/s}$ edges. 
\end{thm}

Here we restate Theorem \ref{oddgirthpluseven}.

\begin{thm*}
For any $k \ge l+1$, we have
$$\Omega(n^{1+\frac{1}{2k+1}}) = ex(n, C_{2l+1}, \cC_{2l} \cup \{C_{2k}\}) = O(n^{1+\frac{l}{l+1}}).$$
\end{thm*}

\begin{proof}
For the lower bound, consider a $(2l+1)$-uniform hypergraph of girth $2k+1$ with $n^{1+1/(2k+1)}$ hyperedges (guaranteed by Theorem \ref{Nes_Rodl}) and then replace each hyperedge by a copy of $C_{2l+1}$. It is easy to check that the resulting graph does not contain any cycle of length at most $2k$ except $2l+1$.

Now we prove the upper bound. Consider a $\cC_{2l} \cup \{C_{2k}\}$-free graph $G$. 
Since all the cycles of length at most $2l$ are forbidden, for any vertex $v$ in $G$, there are no edges inside $N_i(v)$ for $i<l$, and the number of cycles of length $2l+1$ containing $v$ is equal to the number of edges in $N_l(v)$. For a neighbor $w$ of $v$ let $Q(v,w)=N_l(v)\cap N_{l-1}(w)$. 

\begin{claim}\label{lini} For any vertex $v \in V(G)$, there exists a constant $c=c(k,l)\ge 2$ such that there are at most

\smallskip 

\textbf{(i)} $c|N_l(v)|$ edges inside $N_l(v)$, and

\smallskip 

\textbf{(ii)} $c(|N_l(v)|+|N_{l+1}(v)|)$ edges between $N_l(v)$ and $N_{l+1}(v)$.
\end{claim}

\begin{proof} For \textbf{(i)} let $w_1,\dots, w_{s-1}$ be the neighbors of $v$, and let $V_i=Q(v,w_i)$ for $1 \le i \le s-1$. This gives a partition of $N_l(v)$. Observe that an edge inside $V_i$ would create a cycle of length at most $2l-1$, as both its vertices are connected to $w_i$ with a path of length $l-1$. Similarly a $P_3$ with both endpoints inside $V_i$ would create a cycle of length at most $2l$. Finally, a $P_{2k-2l+1}$ with endpoints in $V_i$ and $V_j$ would create a cycle of length $2k$ together with the two (internally disjoint) paths of $l$ connecting its endpoints to $v$. Thus we can apply Lemma \ref{utso} to finish the proof.

For \textbf{(ii)} we add $V_s=N_{l+1}(v)$ to the family of sets $V_i$, $1 \le i \le s-1$ defined before, and delete the edges inside $V_s$, as well as the edges inside $N_l(v)$. Observe that if a $P_{2k-2l+1}$ has endpoints in different parts $V_i$ and $V_j$, then $i\neq s\neq j$ because of the parity of the length of the path. Thus we can apply Lemma \ref{utso} to finish the proof.
\end{proof}

Now we delete every vertex which is contained in at most $4c^{l+1}n^{\frac{l}{l+1}}$ copies of $C_{2l+1}$ from $G$, and we repeat this procedure until we obtain a graph $G'$ where every vertex is contained in more than $4c^{l+1}n^{\frac{l}{l+1}}$ copies of $C_{2l+1}$. We claim that $G'$ has at most $O(n^{1+\frac{l}{l+1}})$ copies of $C_{2l+1}$. As we deleted at most $O(n^{\frac{l}{l+1}})$ $C_{2l+1}$'s with every vertex, this will finish the proof.

Assume $G'$ contains more than $cn^{1+\frac{l}{l+1}}$ copies of $C_{2l+1}$. First we show that the maximum degree in $G'$ is at least $cn^{\frac{1}{l+1}}$. Indeed, otherwise $N_i(v)$ contains at most $c^in^{\frac{i}{l+1}}$ vertices for every $1 \le i \le l$ (here we use that $G'$ is $\cC_{2l}$-free), thus there are at most $c^ln^{\frac{l}{l+1}}$ vertices in $N_l(v)$, hence there are at most $c^{l+1}n^{\frac{l}{l+1}}$ edges inside $N_l(v)$ by \textbf{(i)} of Claim \ref{lini}, which means $v$ is contained in at most $c^{l+1}n^{\frac{l}{l+1}}$ copies of $C_{2l+1}$, so it should have been deleted, a contradiction.

Thus we can assume there is a vertex $v$ of degree at least $cn^{\frac{1}{l+1}}$. We will show that either $v$ or one of its neighbors is contained in at most $c^ln^{\frac{l}{l+1}}$ copies of $C_{2l+1}$. For a neighbor $w$ of $v$ let $S_0(w)=N_{l}(w)\cap N_{l-1}(v)$, $S_1(w)=N_{l}(w)\cap N_{l}(v)$ and $S_2(w)=N_l(w) \cap N_{l+1}(v)$. Notice that $N_l(w)=S_0(w)\cup S_1(w)\cup S_2(w)$.

Let us sum up the number of edges $pq$ with $p \in Q(v,w)$ and $q \in N_l(v)\cup N_{l+1}(v)$, over all the neighbors $w$ of $v$. This way we counted every edge inside $N_l(v)$ or between $N_l(v)$ and $N_{l+1}(v)$ at most twice; moreover the number of such edges is at most $2cn$ by Claim \ref{lini}. Therefore, the total sum is at most $4cn$. As $d(v)\ge cn^{\frac{1}{l+1}}$, $v$ has a neighbor $w$ such that there are at most $4n^{\frac{l}{l+1}}$ edges between vertices in $Q(v,w)$ and vertices in $N_l(v)\cup N_{l+1}(v)$. This also means $|S_1(w)\cup S_2(w)|\le 4n^{\frac{l}{l+1}}$. 

We claim that there are at most $(c+1)n^{\frac{l}{l+1}}$ edges inside $N_l(w)=S_0(w)\cup S_1(w)\cup S_2(w)$. There is no edge inside $S_0(w)$ as there is no edge inside $N_{l-1}(v)$. There is no edge between $S_0(w)$ and $S_2(w)$, since otherwise its endpoint in $S_2(w)$ would have to be in $N_l(v)$. A vertex $u \in S_1(w)$ is connected to at most one vertex in $S_0(w)$, otherwise we would obtain two distinct paths of length $l$ between $u$ and $v$, giving us a cycle of length at most $2l$. Hence the number of edges inside $N_l(w)$ incident to elements of $S_0(w)$ is at most $|S_1(w)|\le 4n^{\frac{l}{l+1}}$.

Let us now partition $N_l(w)$ into sets $Q(w,w') = N_l(w) \cap N_{l-1}(w')$ for each neighbor $w'$ of $w$.
Observe that $Q(w,v)=S_0(w)$. For the remaining $d(w)-1$ parts we want to apply Lemma \ref{utso} similarly to \textbf{(i)} of Claim \ref{lini}. In fact, by deleting $S_0(w)$ we obtain another graph $G''$ where the same cycles are forbidden and the $l$--{th} neighborhood of $w$ is $S_1(w)\cup S_2(w)$. Thus applying Claim \ref{lini} we obtain that there are at most $c(|S_1(w)\cup S_2(w)|) \le 4cn^{\frac{l}{l+1}}$ edges inside $S_1\cup S_2$.

Thus altogether there are at most $4(c+1)n^{\frac{l}{l+1}}<4c^{l+1}n^{\frac{l}{l+1}}$ edges inside $N_l(w)$ (where $c >0$ is a constant chosen so that the previous inequality is satisfied), hence $w$ should have been deleted, a contradiction.

\end{proof}

Here we restate Theorem \ref{oddgirthplusodd} and prove it.

\begin{thm*} For $k > l \ge 2$ we have
$$\Omega(n^{1+\frac{1}{2k+2}}) = ex(n, C_{2l+1}, \cC_{2l} \cup \{C_{2k+1}\}) = O(n^2).$$ 
\end{thm*}

\begin{proof}
For the lower bound, consider a $(2l+1)$-uniform hypergraph of girth $2k+2$ with $n^{1+1/(2k+2)}$ hyperedges and then replace each hyperedge by a copy of $C_{2l+1}$.

Now we prove the upper bound. Let $v$ be an arbitrary vertex in a $\cC_{2l} \cup \{C_{2k+1}\}$-free graph $G$. We will upper bound the number of $C_{2l+1}$'s containing $v$.
There are no edges inside $N_i(v)$ for each $i < l$. Indeed, if there is an edge then we can find a forbidden short odd cycle containing that edge (because the end points of that edge have a common ancestor). This shows that every $C_{2l+1}$ containing $v$ must use an (actually exactly one) edge from $N_l(v)$. So the number of $C_{2l+1}$'s containing $v$ is upper bounded by the number of edges in $N_l(v)$.  We claim the following. 

\begin{claim}
\label{C2k+1scontainingv}
The number of edges in $N_l(v)$ is $O(|N_l(v)|) = O(n)$. 
\end{claim}


\begin{proof}[Proof of Claim] Color each vertex in $N_l(v)$ by its (unique) ancestor in $N_1(v)$. Then the resulting color classes $A_1, A_2, \ldots A_t$ partition $N_l(v)$.
There are no edges inside the color classes, because such an edge would be contained in a forbidden short odd cycle. 

One can partition the color classes into two parts $\{A_i \mid i \in I\}$ and $\{A_i \mid i \in J\}$ (with $I \cup J = \{1,2, \ldots, t\}$ and $I \cap J = \emptyset$), so that at least half of all the edges in $N_l(v)$ are between the vertices of the two parts. 
Now as $C_{2k+1}$ is forbidden, there is no path of length $2k+1-2l$ between the two parts, as such a path would have its end vertices in different classes. (Note that here we use that the parity of the path length $2k+1-2l$ is odd.)
Thus by Erd\H{o}s-Gallai theorem there are only at most $O(|N_l(v)|) = O(n)$ edges between the two parts. This implies that the total number of edges in $N_l(v)$ is at most twice as many, completing the proof of the claim.
\end{proof}

So using Claim \ref{C2k+1scontainingv}, the number of $C_{2l+1}$'s containing any fixed vertex $v$ is $O(n)$. Thus the total number of $C_{2l+1}$'s in $G$ is at most $O(n^2)$, as desired. This completes the proof. 
\end{proof}


We conjecture that even if the additional forbidden cycle has odd length, we can only have a sub-quadratic number of $C_{2l+1}$'s. See Conjecture \ref{oddandodd} stated in the last section.

\section{Number of copies of $P_l$ in a graph avoiding a cycle of given length}

In the first subsection, we will consider the case when an even cycle is forbidden and prove Theorem \ref{pathevenupper} and Theorem \ref{pathevenlower}, and in the next subsection, we will deal with the case when an odd cycle is forbidden and prove Theorem \ref{pathodd}. 

\subsection{Bounds on $ex(n, P_l, C_{2k})$}
For the upper bound, we use a spectral method similar to the one used in \cite{GySTZ2018}. 

The spectral radius of a finite graph is defined to be the spectral radius of its adjacency matrix. Given a graph $G$, the spectral radius of $G$ is denoted by $\mu(G)$. Given a matrix $A$, its spectral radius is denoted by $\mu(A)$. If $A$ is the adjacency matrix of $G$, then of course, $\mu(G) = \mu(A)$.

Nikiforov \cite{Nikiforov_even} showed the following. 
\begin{thm}[Nikiforov]
\label{Nikiforov}
Let $G$ be a $C_{2k}$-free graph on $n$ vertices. Then for any $k \ge 1$, we have $$\mu(G) \le \frac{k-1}{2}+\sqrt{(k-1)n}+o(n).$$ 
\end{thm}

Note that in the case $k=2$, a sharper bound is known: The maximum spectral radius of a $C_{4}$-free graph on $n$ vertices is $\frac{1}{2}+ \sqrt{n-3/4}+O(1/n)$, where for odd $n$ the $O(1/n)$ term is zero. Now we prove Theorem \ref{pathevenupper}, restated below.

\begin{thm*}
We have $$ex(n, P_l, C_{2k}) \le (1+o(1)) \frac{1}{2} (k-1)^{\frac{l-1}{2}}n^{\frac{l+1}{2}} .$$
\end{thm*}
\begin{proof}
Let $A$ be the adjacency matrix of a $C_{2k}$-free graph $G$. Recall that $\cN(P_l, G)$ denotes the number of copies of $P_l$ in $G$. Let $\cN(W_l, G)$ denote the number of walks consisting of $l$ vertices in $G$. Note that $2\cN(P_l, G) \le \cN(W_l, G)$, since every path corresponds to two walks.

Then we have,
\begin{equation}
\label{spectral1}
\frac{2\cN(P_l, G)}{n} \le \frac{\cN(W_l, G)}{n} = \frac{\mathbf{1}^t A^{l-1} \mathbf{1}}{\mathbf{1}^t \mathbf{1}}
\end{equation}

Note that $\mathbf{1}$ is the column vector with all entries being 1. The right-hand-side of \eqref{spectral1} is at most $\mu(A^{l-1})$ because the spectral radius of any Hermitian matrix $M$ is the supremum of the quotient $\frac{x^*Mx}{x^*x}$, where $x$ ranges over $\mathbb{C}^n \backslash \{0\}$. Moreover, using Theorem \ref{Nikiforov}, we have $$\mu(A^{l-1}) = (\mu(A))^{l-1} = (\mu(G))^{l-1} \le \left(\frac{k-1}{2}+\sqrt{(k-1)n}+o(n) \right)^{l-1} \le (1+o(1)) ((k-1)n)^{\frac{l-1}{2}},$$ completing the proof.
\end{proof}




Now we provide some lower bounds on $ex(n, P_l, C_{2k})$.

\subsection*{Constructing $C_{2k}$-free graphs with many copies of $P_l$}

We prove Theorem \ref{pathevenlower}. Note that the behavior of the extremal function seems to be very different in the cases $l < 2k$ and $l \ge 2k$. 

\begin{thm*}
If $2 \le l < 2k$, then $$ex(n, P_l, C_{2k}) \ge  (1+o(1))\frac{1}{2} (k-1)_{\lfloor \frac{l}{2} \rfloor} n^{\lceil \frac{l}{2} \rceil}.$$
If $l \ge 2k$, then $$ex(n, P_l, C_{2k}) \ge (1+o(1)) \max \left \{\left(\frac{n}{\lfloor{l/2}\rfloor}\right)^{\lceil l/2 \rceil}, \left(\frac{(k-1)}{4(k-2)^{k+2}}\right)^{\lceil \frac{l}{2} \rceil} (k-1)_{\lfloor \frac{l}{2} \rfloor}n^{\lceil \frac{l}{2} \rceil} \right \}.$$
\end{thm*}

\begin{proof}
In the case $l < 2k$, we take a complete bipartite graph $B$ with parts of size $k-1$ and $n-(k-1)$. 
Clearly, $B$ is $C_{2k}$-free and the number of copies of $P_l$ in $B$ is at least $$\frac{1}{2} (k-1)_{\lfloor \frac{l}{2} \rfloor}(n-(k-1))_{\lceil \frac{l}{2} \rceil} = \frac{1}{2} (k-1)_{\lfloor \frac{l}{2} \rfloor} n^{\lceil \frac{l}{2} \rceil} (1+o(1)).$$

Now we consider the case $l \ge 2k$. First we give a simple construction. Consider a path $v_1v_2\ldots v_l$ and for each odd $i$, replace the vertex $v_i$ by $b$ vertices $v_i^1, v_i^2, \ldots, v_i^b$ where each of them is adjacent to the same vertices that $v_i$ was adjacent to. Choose $b = \frac{n-\lfloor{l/2}\rfloor}{\lfloor{l/2}\rfloor} = (1+o(1))\frac{n}{\lfloor{l/2}\rfloor}$. The resulting graph only contains cycles of length 4, so it is $C_{2k}$-free as long as $k \not = 2$. Moreover, it contains at least $$(1+o(1)) b^{\lceil l/2 \rceil} = (1+o(1)) \left(\frac{n}{\lfloor{l/2}\rfloor}\right)^{\lceil l/2 \rceil}$$ copies of $P_l$. The case $k=2$ is dealt with in Theorem \ref{celebrated}

Now we give a different construction which gives a better lower bound when $l$ is large compared to $k$.  We will use the following theorem of Ellis and Linial \cite{Ellis_Linial}.

\begin{thm}[Ellis, Linial \cite{Ellis_Linial}]
\label{Ellis_linial}
Let $r$,$d$ and $g$ be integers with $d \ge 2$ and $r,g \ge 3$. Then there exists an $r$-uniform, $d$-regular hypergraph $\mathcal H$ with girth at least $g$, and at most
$$n_r(g,d) := (r-1) \left( 1+d(r-1) \frac{(d-1)^g(r-1)^{g}-1}{(d-1)(r-1)-1} \right) < 4((d-1)(r-1))^{g+1}$$
vertices.
\end{thm}

Let $g =k+1$ and $r= k-1$. Consider the hypergraph $\mathcal H$ given by Theorem \ref{Ellis_linial} with 
\begin{equation}
\abs{V(\mathcal H)} \le n_r(g,d) = n_{k-1}(k+1, d).
\end{equation}

Notice that the number of hyperedges in $\mathcal H$ is 

\begin{equation}
\label{edgelowerplc2k}
\abs{E(\mathcal H)} \le \frac{d \cdot n_r(g,d)}{r} = \frac{d \cdot n_{k-1}(k+1, d)}{k-1}.
\end{equation}

Let $E(\mathcal H) = \{h_1, h_2, \ldots, h_m \}$. To each hyperedge $h_i \in E(\mathcal H)$, we add a set $S_i$ of new vertices with $$\abs{S_i} = \frac{(n-\abs{V(\mathcal H)})}{\abs{E(\mathcal H)}}$$ (note for $i \not = j$, we take $S_i \cap S_j = \emptyset$). Now we construct a graph $G$ as follows: For each $i$ with $1 \le i \le m$, consider the sets $h_i$, $S_i$ and add all possible edges between $h_i$ and $S_i$. That is, $E(G) = \{uv \mid u \in h_i, v \in S_i \text{ for some } 1 \le i \le m \}$. It is easy to check that $G$ is $C_{2k}$-free. Note that $G$ is a bipartite graph with parts $U:= V(\mathcal H)$ and $D:= \cup_{i=1}^m S_i$. Moreover, the degree of every vertex of $G$ in $U$ is $d$ times the size of a set $S_i$, so it is $$\frac{d (n-\abs{V(\mathcal H)})}{\abs{E(\mathcal H)}}.$$ And the degree of every vertex in $D$ is the size of a set $h_i$, so it is $k-1$. Therefore, the number of copies of $P_l$ in $G$ is at least

$$ \left(\frac{d (n-\abs{V(\mathcal H)})}{\abs{E(\mathcal H)}}\right)_{\lceil \frac{l}{2} \rceil} (k-1)_{\lfloor \frac{l}{2} \rfloor} = (1+o(1))\left(\frac{d n}{\abs{E(\mathcal H)}}\right)^{\lceil \frac{l}{2} \rceil} (k-1)_{\lfloor \frac{l}{2} \rfloor}.$$
Using \eqref{edgelowerplc2k}, this is at least 
$$ (1+o(1)) \left(\frac{(k-1)n}{n_{k-1}(k+1, d)}\right)^{\lceil \frac{l}{2} \rceil} (k-1)_{\lfloor \frac{l}{2} \rfloor}.   $$

Choosing $d = 2$ and using Theorem \ref{Ellis_linial}, we have $$n_{k-1}(k+1, d) = (k-2) \left( 1+2(k-2) \frac{(k-2)^{k+1}-1}{(k-2)-1} \right) < 4(k-2)^{k+2}.$$

So, $ex(n, P_l, C_{2k})$ is at least $$ (1+o(1))\left(\frac{(k-1)n}{n_{k-1}(k+1, 2)}\right)^{\lceil \frac{l}{2} \rceil} (k-1)_{\lfloor \frac{l}{2} \rfloor} >  (1+o(1)) \left(\frac{(k-1)}{4(k-2)^{k+2}}\right)^{\lceil \frac{l}{2} \rceil} (k-1)_{\lfloor \frac{l}{2} \rfloor}n^{\lceil \frac{l}{2} \rceil}. $$
\end{proof}

\subsection{Bounds on $ex(n, P_l, C_{2k+1})$}
For the upper bound we will again use a spectral bound. We will use the following theorem of Nikiforov \cite{Nikiforov_odd}.
\begin{thm}[Nikiforov]
\label{Nikiforov2}
Let $G$ be a $C_{2k+1}$-free graph on $n$ vertices. Then for any $k \ge 1$ and $n > 320(2k+1)$, we have $$\mu(G) \le \sqrt{n^2/4}.$$ 
\end{thm}

Now we prove Theorem \ref{pathodd}, restated below.

\begin{thm*}
We have $$ex(n, P_l, C_{2k+1}) = (1+o(1))\left(\frac{n}{2}\right)^{l}.$$
\end{thm*}

\begin{proof}
For the lower bound, consider a complete bipartite graph $B$ with $n/2$ vertices on each side. Then clearly, $B$ does not contain any odd cycle and it contains at least $$(1+o(1)) \left(\frac{n}{2}\right)^{l}$$ copies of $P_l$.

The proof of the upper bound is similar to that of the proof of Theorem \ref{pathevenupper}. Let $A$ be the adjacency matrix of a $C_{2k+1}$-free graph $G$. 
Then, for $n$ large enough, using Theorem \ref{Nikiforov2} we get,
\begin{equation*}
\label{spectral2}
\frac{2\cN(P_l, G)}{n} \le \frac{\cN(W_l, G)}{n} = \frac{\mathbf{1}^t A^{l-1} \mathbf{1}}{\mathbf{1}^t \mathbf{1}} \le \mu(A^{l-1}) = (\mu(A))^{l-1} = (\mu(G))^{l-1} \le \left(\sqrt{\frac{n^2}{4}}\right)^{l-1} = \left(\frac{n}{2}\right)^{l-1}.
\end{equation*}

Thus, $$\cN(P_1, G) \le \left(\frac{n}{2}\right)^{l}, $$ completing the proof of the theorem.
\end{proof}

\begin{remark}
Note that the number of copies of $P_{2l}$ in a graph $G$ is at least $2l$ times the number of copies of $C_{2l}$ in $G$. Indeed, every copy of $C_{2l}$ contains $2l$ copies of $P_{2l}$. Moreover, a copy of $P_{2l}$ belongs to at most one copy of $C_{2l}$.
Thus Theorem \ref{pathodd} implies Theorem \ref{evenodd}.
\end{remark}

\section{Concluding remarks and questions}
\label{concludingremarks}

We finish our article by posing some questions. 

\ 

$\bullet$ Naturally, it would be interesting to prove asymptotic or exact results corresponding to our results where we only know the order of magnitude. For example it would be nice to close the gap between the lower and upper bounds in Theorem \ref{main2}.


\ 

\noindent 
We also pose some conjectures when a family of cycles are forbidden.

\ 

$\bullet$  We proved in Theorem \ref{longer_kor}, that for any $k > l$ and $m\ge 2$ such that $2k \neq ml$ we have $$ex(n,C_{ml},\cC_{2l-1} \cup \{C_{2k}\})=\Theta(n^m).$$

We conjecture that it is true for longer cycles as well.

\begin{conjecture}
 For any $k > l$, $m\ge 2$ and $1 \le j < l$ with $ml+j \neq 2k$ we have $$ex(n,C_{ml+j},\cC_{2l-1} \cup \{C_{2k}\})=\Theta(n^m).$$
    \end{conjecture}

$\bullet$ We prove in Theorem \ref{oddgirthplusodd} that for $l > k \ge 2$ we have
$$ex(n, C_{2k+1}, \cC_{2k} \cup \{C_{2l+1}\}) = O(n^2).$$ 

However, we conjecture that the truth is smaller.
\begin{conjecture}\label{oddandodd}
For any integers $k < l$, there is an $\epsilon > 0$ such that 
$$ex(n, C_{2k+1}, \cC_{2k} \cup \{C_{2l+1}\}) = O(n^{2-\epsilon}).$$
\end{conjecture}

The following theorem supports Conjecture \ref{oddandodd}.

\begin{thm} We have
\label{supporting_conjecture}
$$ex(n, C_5, \cC_{4} \cup \{C_9\})=O(n^{11/12}).$$
\end{thm}





\begin{proof} Let us consider a $\cC_{4} \cup \{C_9\} = \{C_3,C_4,C_9\}$-free graph $G$. First we delete every edge that is contained in less than $17$ $C_5$'s, 
then repeat this until every edge is contained in at least $17$ $C_5$'s. We have deleted at most $17|E(G)|=O(n^{3/2})$ 
$C_5$'s this way (note that $|E(G)|=O(n^{3/2})$ follows from the fact that $G$ is $C_4$-free). Let $G'$ be the graph obtained this way.

Observe that if a five-cycle $C := v_1v_2v_3v_4v_5v_1$ shares the edge $v_1v_2$ with another $C_5$, then they either share 
also the edge $v_2v_3$ or $v_5v_1$ and no other vertices, or they share only the edge $v_1v_2$. If there are at least six five-cycles sharing only $v_1v_2$ with $C$, we say $v_1v_2$ is an \textit{unfriendly} edge for $C$, otherwise it is called a \textit{friendly} edge for 
$C$. Our plan is to show first that a $C_5$ cannot contain both friendly and unfriendly edges, then using this we will show that a $C_5$ cannot contain friendly edges. Thus every edge is unfriendly for every $C_5$, and this will imply that $G'$ is $C_6$-free.

Assume $C$ contains both friendly and unfriendly edges. Then it is easy to
see that it contains an unfriendly edge, say $v_1v_2$, and a path $P$ of two edges not containing $v_1v_2$ such that
$C$ shares $P$ with a set $\cS$ of at least $6$ other $C_5$'s. (Note that the cycles in $\mathcal S$ only share $P$.) Now there is a cycle $v_1v_2w_3w_4w_5v_1$ by the 
unfriendliness of $v_1v_2$ that contains three new vertices $w_3,w_4,w_5$. Then we replace $v_1v_2$ in $C$ with
$v_1w_3w_4w_5v_2$ to obtain a $C_8$. Afterwards, there is a cycle in $\cS$ that does not contain any of $w_3$, $w_4$ 
and $w_5$ as the elements of $\cS$ are  vertex disjoint outside $C$. Thus we can replace $P$ in this $C_8$
with a path of three edges to obtain a $C_9$, a contradiction.


Assume now $C$ contains only friendly edges. The edge $v_1v_2$ is contained in at least $6$ other $C_5$'s 
together with one of its two neighboring edges, say $v_2v_3$. At least of these $6$ $C_5$'s does not contain the vertices $v_4$ and $v_5$, let it be $v_1v_2v_3w_1w_2v_1$. Thus replacing the two-edge path $v_1v_2v_3$ with the three-edge path $v_3w_1w_2v_1$ to obtain the six-cycle $v_4v_5v_1w_2w_1v_3v_4$. The 
edge $w_1w_2$ is friendly for $v_1v_2v_3w_1w_2v_1$ (because otherwise, it would contain both friendly and unfriendly edges). Thus $w_1w_2$ is in at least $6$ other $C_5$'s together with either $v_3w_1$ or $w_2v_1$. In the same way as before, we can replace this two-edge path with a three-edge path to obtain a seven-cycle. Repeating this procedure we can obtain an eight-cycle and then a nine-cycle, a contradiction.
Indeed, at each step, we are given a cycle $C'$ of length between 5 and 8, and we add two new vertices to it in place of one of its vertices by replacing a two-edge path $P$ with a three-edge path to increase the length of $C'$. We have to make sure that
the two new vertices are disjoint from the other vertices of $C'$. Since there are $6$ $C_5$'s containing $P$ which are vertex-disjoint outside $P$, it is easy to find a $C_5$ that avoids the at most 5 vertices of $C'$ outside $P$.

Hence every edge is unfriendly to every $C_5$ in $G'$. Then we claim that there is no $C_6$ in $G'$. Indeed, otherwise we 
consider an arbitrary edge $uv$ of that $C_6$, there is a set $\cS'$ of at least $17$ $C_5$'s that each contain $uv$. Because of the unfriendliness of $uv$ to each of the cycles in $\cS'$, they do not share any other vertices except $u$ and $v$, so at least one of them is disjoint from the other vertices of the $C_6$, thus we can exchange $e$ to a $4$-edge-path in the $C_6$, obtaining a $C_9$, a contradiction.


We obtained that after deleting $O(n^{3/2})$ edges, the resulting graph $G'$ is $\{C_3,C_4,C_6\}$-free, 
thus it contains at most $O(n^{11/12})$ $C_5$'s by Theorem \ref{oddgirthpluseven}.

\end{proof}

\subsection*{Remarks about $ex(n,C_{l},\cC_A)$ for a given set $A$ of cycle lengths} 

After the investigation carried out in this article it is natural to ask to determine $ex(n,C_{l},\cC_{A})$ for any set $A$.

Let us note that the behavior of $ex(n,C_{l},\cC_A)$ is more complicated if $l$ is not $4$ or $6$. 
A simple construction of a $\cC_A$-free graph $G$ is the following. Let $2r$ be the shortest length of an even cycle which is allowed (note that if no even cycle is allowed, then the total number of cycles is $O(n)$ by a theorem in \cite{GKPP2016}). Let $p=\lfloor l/r \rfloor$. If $r$ divides $l$, then the theta-$(n,C_p,r)$ graph contains $\Omega(n^{p})$ copies of $C_l$, some $C_{2r}$'s and no other cycles.
 If $r$ does not divide $l$, it is easy to see that we can add a path
with $l-pr$ new vertices between the two end vertices of a theta-$(n-(l-pr),P_{p+1},r)$ graph to obtain a graph with $\Omega(n^{p})$ many $C_l$'s, some $C_{2r}$'s and no other cycles.


Observe that in these cases, we still have an integer in the exponent. However, Theorem \ref{solymosiwong} (by Solymosi and Wong) shows that if $l \ge 4$ is even, then $ex(n,C_{2l},\cC_6) = \Theta(n^{l/3})$ since it is known that Erd\H os's Girth Conjecture holds for $m = 3$. This shows an example where the exponent is not an integer. 

The situation is even more complicated when $l$ is odd. Let us examine the simplest case $l=5$, i.e. $ex(n,C_5,\cC_A)$. If $A$ contains only one element, Gishboliner and Shapira \cite{gs2017} determined the order of magnitude (it is $0$ or $n^2$ or $n^{5/2}$). If there are at least two elements in $A$ but $4\not\in A$, then the construction described above gives $ex(n,C_5,\cC_A)=\Omega(n^2)$, while the result of Gishboliner and Shapira \cite{gs2017} implies $ex(n,C_5,\cC_A)=O(n^2)$. If $A=\{C_3,C_4\}$, then Lemma \ref{paratlan} shows $ex(n,C_5,\{C_3,C_4\})=\Theta(ex(n,C_5,C_4))$, which is $\Theta(n^{5/2})$ by Theorem \ref{celebrated}. What remains is the case $A$ contains 4 and another number. In this case Theorem \ref{oddgirthpluseven} and Theorem \ref{oddgirthplusodd} give some bounds that are not sharp.

\section*{Acknowledgements}

We are grateful to an anonymous referee for carefully reading our paper and for their helpful remarks.
	
   \vspace{2mm}
	Research of Gerbner was supported by the J\'anos Bolyai Research Fellowship of the Hungarian Academy of Sciences and by the National Research, Development and Innovation Office -- NKFIH, grant K 116769.
	
	\vspace{2mm}
	Research of Gy\H ori and  Methuku was supported by the National Research, Development and Innovation Office -- NKFIH, grant K 116769 and SNN 117879.
	
	\vspace{2mm}
	Research of Vizer was supported by the National Research, Development and Innovation Office -- NKFIH, grant SNN 116095 and K 116769.

\end{document}